\documentclass[a4paper,12pt, reqno]{amsart}
\usepackage{a4}
\usepackage{amsfonts,exscale}
\usepackage{mleftright}

\usepackage{srcltx,hyperref}
\usepackage{cite}
\usepackage{graphics}
\usepackage{color}

\long\def \forget#1{}

\forget{
\pdfoutput=1

\usepackage[pdftex]{hyperref}
\hypersetup{%
pdftitle ={Large sieve inequality},
pdfauthor = {Rajneesh Kumar singh, Stephan Baier},
pdfcreator = {pdfLatex},
pdfproducer={pdfLatex with hyperref}
}
}

\usepackage{amsmath}
\usepackage{amssymb}
\usepackage{amscd}
\usepackage{textcomp}
\usepackage{amsthm}
\theoremstyle{plain}
\newtheorem{Lemma}{Lemma}[section]
\newtheorem{Theorem}[Lemma]{Theorem}
\newtheorem{Proposition}[Lemma]{Proposition}

\theoremstyle{definition}

\newcommand{\DS}{\displaystyle}
\newcommand{\TS}{\textstyle}

\let\setminus\smallsetminus

\newcommand{\es}{\enspace}

\newcommand{\invlim}[1][]{\ifthenelse{\equal{#1}{}}
{\DS \lim_{\longleftarrow}}
{\DS \lim_{\underset{#1}{\longleftarrow}}}
}
\newcommand{\dirlim}[1][]{\ifthenelse{\equal{#1}{}}
{\DS \lim_{\longrightarrow}}
{\DS \lim_{\underset{#1}{\longrightarrow}}}
}

\newcommand{\ol}[1]{{\overline{#1}}}

\newcommand{\BOne} {{\mathchoice{\hbox{\rm1\kern-2.7pt l\kern.9pt}}
{\hbox{\rm1\kern-2.7pt l\kern.9pt}}
{\hbox{\scriptsize\rm1\kern-2.3pt l\kern.4pt}}
{\hbox{\scriptsize\rm1\kern-2.4pt l\kern.5pt}}}}

\newcommand{\BC}{{\mathbb{C}}}

\newcommand{\BF}{{\mathbb{F}}}

\newcommand{\BN}{{\mathbb{N}}}

\newcommand{\BT}{{\mathbb{T}}}

\usepackage{mathrsfs}

\newcommand{\CS}{{\mathcal{S}}}

\DeclareMathOperator{\cha}{char}

\DeclareMathOperator{\ord}{ord}
\DeclareMathOperator{\Tr}{Tr}

\DeclareMathOperator{\md}{{mod}}
\begin{document}
\author{Stephan Baier, Rajneesh Kumar Singh}
\address{Stephan Baier, Ramakrishna Mission Vivekananda Educational Research Institute, G. T. Road, PO Belur Math, Howrah, West Bengal 711202, 
India; email: email$_{-}$baier@yahoo.de}
\address{Rajneesh Kumar Singh, Ramakrishna Mission Vivekananda Educational Research Institute, G. T. Road, PO Belur Math, Howrah, West Bengal 711202, India; rajneeshkumar.s@gmail.com}
\title{THE LARGE SIEVE WITH SQUARE MODULI IN FUNCTION FIELDS}
\maketitle
\begin{abstract} We prove a lower and an upper bound for the large sieve with square moduli in function fields. These bounds correspond to bounds for the classical large 
sieve with square
moduli established in \cite{BaLyZh} and \cite{BaZh}. Our lower bound in the function field setting contradicts an upper bound obtained in \cite{BaSi}. 
Indeed, in \cite{BaSierr} 
we pointed out an error in \cite{BaSi}. 
\end{abstract}
\tableofcontents
\bigskip

\newpage
\title{}
\maketitle

\section{Introduction}
In \cite{BaSi}, we investigated the large sieve with restricted sets of moduli (in particular, power moduli) for function fields. Unfortunately, as pointed out in
\cite{BaSierr}, there is a serious error 
in this paper. Indeed, in section 3 of the present paper, we shall derive a lower bound for the large sieve with square moduli  
which contradicts an upper bound obtained in \cite{BaSi}. 

The plan for this paper is as follows.
In the next section, we shall provide general notations. A counterexample to \cite[Corollary 5.2.]{BaSi} for the case of 
square moduli will be given in
section 3. This is the function field analog of a lower bound for the classical large sieve with square moduli in 
\cite{BaLyZh}. From section 4 onward, we shall start a thorough investigation of the large sieve with square moduli for function fields. 
Along the lines in \cite{BaZh}, which deals with the classical case, we shall obtain a corresponding result for the function field case. 

\section{Notations} 
This section is essentially copied from \cite[section 2]{BaSi}, but here we confine ourselves to the one-dimensional case. We collect together notations 
and preliminaries mostly drawn from \cite{Hsu}. 

Throughout this paper, we assume that $q$ is an odd prime power and hence the characteristic of $\BF_q$ is not equal to 2. 
Let $\BF_q$ be a fixed finite field with $q$ elements of characteristic $p$ and let $\Tr : \BF_q \to \BF_p$ be the trace map.
Let $\BF_q(t)_{\infty}$ be the completion of $\BF_q(t)$ at $\infty$ (i.e. $\BF_q(t)_{\infty}=\BF_q(( 1/t))$).
The absolute value $|\cdot |_{\infty}$ on $\BF_q(t)_{\infty}$ is defined by $|0|_{\infty}:=0$ and 
\begin{align*}
\mathrel \bigg | \sum\limits_{i= -\infty}^n a_i t^i\bigg| := q^n, \es \text{if} \es 0 \neq a_n \in \BF_q.
\end{align*}
We endow the torus $\BT = \BF_q(t)_\infty/\BF_q[t]$ with a metric 
\begin{equation} \label{metric}
\|f\| := \inf_{f'\sim f}|f'|_\infty, 
\end{equation}
where $f'\sim f$ means that $f'-f \in \BF_q[t]$. Note that for all $f\in \BT$, we have $\|f\|\leq 1/q$. We also
define the fractional part by
$$
\left\{ \sum\limits_{i= -\infty}^n a_i t^i\right\} :=\sum\limits_{i= -\infty}^{-1} a_i t^i
$$
and note that
$$
\| f \|=|\{f\}|_{\infty}
$$
for all $f\in \BF_q(t)_{\infty}$. 

The non-trivial additive character $E: \BF_q \to \BC^\times$ is defined by 
\begin{align*}
E(x) = \exp \bigg\{\frac{2\pi i}{p} \Tr(x)\bigg\},
\end{align*}
and the map $e: \BF_q(t)_\infty \to \BC^\times$ is defined by 
\begin{align*}
e\bigg(\sum\limits_{i= -\infty}^n a_i t^i\bigg) = E(a_{-1}).
\end{align*}
This map $e$ is a non-trivial additive character for $\BF_q(t)_\infty$.

Moreover, we denote the ball with center $x$ and radius $q^N$ by $\mathcal{B}(x,N)$. 

\section{A counterexample} 
In this section we shall give an example which shows that \cite[Corollary 5.2]{BaSi} is not correct. Below the original statement.\\

{\bf Claim 1:} (Corollary 5.2. in \cite{BaSi}) {\it Let $N,Q\in \mathbb{N}$ and $a_g\in \mathbb{C}$, where $g\in \BF_q[t]$. Let $\mathcal{S}$ 
be a set
of monic polynomials in $\BF[t]$. Then
\begin{align*}
\sum\limits_{\substack{f \in \mathcal{S} \\ \deg f \leq Q}} \sum\limits_{\substack{r \bmod f\\ (r,f)=1}}\mathrel \bigg 
|\sum\limits_{\substack{g \in \BF_q[t]\\ \deg g\le N}} a_g 
e\Big(g\cdot \frac{r}{f}\Big)\bigg |^2
\le \left(q^{N+1}+ \sharp\mathcal{S}\cdot q^{Q-1}\right)\cdot 
 \sum\limits_{\substack{g \in \BF_q[t]\\ \deg g\le N}} |a_g|^2.
\end{align*}}

The following statement on square moduli is an immediate consequence of Claim 1.\\

{\bf Claim 2:} {\it Let $N,Q\in \mathbb{N}$ and $a_g\in \mathbb{C}$, where $g\in \BF[t]$. Then
\begin{align*}
\sum\limits_{\substack{f \in \BF_q[t] \\ \deg f \leq Q\\ f\, \mbox{\scriptsize \rm monic}}} \sum\limits_{\substack{r \bmod f^2\\ (r,f)=1}}\mathrel \bigg 
|\sum\limits_{\substack{g \in \BF_q[t]\\ \deg g\le N}} a_g 
e\Big(g\cdot \frac{r}{f^2}\Big)\bigg |^2
\le \left(q^{N+1}+ q^{3Q-1}\right)\cdot 
 \sum\limits_{\substack{g \in \BF_q[t]\\ \deg g\le N}} |a_g|^2.
\end{align*}}

We now prove the following Theorem which provides a counterexample to Claim 2. 

\begin{Theorem} \label{counterex} For every $\varepsilon>0$, there are infinitely many natural numbers $Q$ such that for suitable 
$N\in \mathbb{N}$ and sequences 
$(a_g)_{g\in \mathbb{F}_q[t]}$ of complex numbers, we have
\begin{equation*}
\begin{split}
\sum\limits_{\substack{f \in \BF_q[t] \\ \deg f \leq Q \\ f\, \mbox{\scriptsize \rm monic}}}\sum\limits_{\substack{r \md  f^2\\  (r,f)=1}}
\mathrel\bigg|\sum\limits_{\substack{g \in \BF_q[t]\\ \deg g\le N}} a_g  e \Big(g\cdot\frac{r}{f^2}\Big) \mathrel\bigg|^2 
\ge  \frac{q^{\mu Q}}{2} \left(q^{N+1}+ q^{3Q-1}\right)\cdot 
 \sum\limits_{\substack{g \in \BF_q[t]\\ \deg g\le N}} |a_g|^2,
\end{split}
\end{equation*}
where 
$$
\mu:= \frac{\log_q 2}{\log_q((1+\varepsilon)Q)}.
$$
\end{Theorem}

{\bf Remark:} In \cite[Corollary 6.5]{BaSi}, a large sieve inequality for $k$-th power moduli was stated which is weaker than Claim 2. 
This was not derived directly from Claim 1 but from a more general $n$-dimensional large sieve inequality. Therein, for the case $k=2$ of 
square moduli, the term $q^{N+1}+ q^{3Q-1}$ in Claim 2 is replaced by $(q+1)^N+(q+1)^{3Q}$. Theorem \ref{counterex} does {\it not} provide
a counterexample to this weaker large sieve inequality for square moduli. Despite the flaw in \cite{BaSi} described in section 1, this weaker
bound may actually be correct (at least, we are not able to come up with a counterexample to it). \\

In the following, let $a_d$ be the number of monic irreducible polynomials  of degree $d$ and let $G$ be their product. Then 
$Q: = \deg G = da_d$. We will denote $G$ by $G_Q$. 
First, we establish the following lower bound for the number of Farey fractions with square denominators near certain elements  of $\BF_q(t)$.

\begin{Lemma}\label{FF} Let $\varepsilon \geq 0$ and let $G_Q$ be as above. Define 
\begin{align}\label{FF2}
\CS(Q): = & \mathrel \Big\{ ( r,f)\in \BF_q[t]^2 : \deg f = Q, \, f\ \mbox{\rm monic}, \, \deg r < 2Q, \\ 
& \mbox{\rm gcd}(r,f)=1, \, \Big|\frac{r}{f^2}- \frac{1}{G_Q}\Big|_\infty\leq \frac{1}{q^{3Q}}\mathrel \Big\}.
\end{align}
Then 
\begin{align}\label{FF1}
\# \CS(Q)\geq  q^{ Q(\log_q2)/\log_q ((1+\varepsilon)Q)}.
\end{align}
\end{Lemma}

In Lemma \ref{FF} above, we write gcd$(r,f)$ instead of the short notation $(r,f)$ because $(r,f)$ is also used to denote a pair in $\BF_q[t]^2$. Throughout the  sequel, if no confusions are possible, we will use the notation $(r,f)$ to denote the greatest common divisor of $r$ and $f$ (which is unique up to units).  

We note that the expected number of Farey fractions of the form $r/f^2$ with $\deg f = Q, \, f\ \mbox{\rm monic}, 
\deg r \leq  Q^2, \, (r,f)=1$ in an interval of length $\Delta$ is, heuristically, of order of magnitude $q^{3Q}\Delta$. 
So the above Lemma \ref{FF} shows that under certain circumstances, the true number can exceed the expectation signi\-ficantly.\\

{\it Proof of Lemma \ref{FF}.} Using the Chinese Remainder Theorem, the number of solution to the congruence 
\[
f^2 \equiv 1 \bmod {G_Q}
\]
with $\deg f =Q$ is $2^{a_d}$. If $f$ solves the above congruence, then 
\[
f^2 = 1 + r G_Q
\]
for some $r$ with $\deg r \leq 2Q$ and $(r,f)=1$, and it follows that 
\[
\mathrel \Big|\frac{r}{f^2}- \frac{1}{G_Q}\mathrel \Big|_\infty\leq \frac{1}{q^{3Q}}.
\]
Hence 
\begin{equation} \label{Sest}
\# \CS(Q)\geq 2^{a_d}. 
\end{equation}
Moreover, using the prime number theorem for polynomials, for any given $\varepsilon >0$,
\[
a_d\geq \frac{q^d}{(1+\varepsilon )d}
\]
if $d$ large enough. 
Therefore,
\[
Q = da_d \geq \frac{q^d}{(1+\varepsilon )},
\]
which gives us  $d \le \log_q((1+\varepsilon)Q)$ and hence
\begin{equation} \label{ad}
a_d=\frac{Q}{d}\ge \frac{Q}{\log_q((1+\varepsilon)Q)}.
\end{equation}
Now the desired inequality \eqref{FF1} follows from \eqref{Sest} and \eqref{ad}.\hfill $\Box$\\ 

Having proved Lemma \ref{FF}, we are ready to prove Theorem \ref{counterex}. Take $G_Q$ as in Lemma \ref{FF}. Further, set 
\[
N : =  3 Q - 2,\, \, \,  a_g : = e\Big(-\frac{g}{G_Q}\Big).
\]
Then 
\[
\sum\limits_{\substack{g \in \BF_q[t]\\ \deg g\le N}}a_g e \Big(g\cdot\frac{r}{f^2}\Big) = \sum\limits_{\substack{g \in \BF_q[t]\\ \deg g\le N}} e(\alpha_g)
\]
with
\[
\alpha_g : = g\Big(\frac{r}{f^2}- \frac{1}{G_Q}\Big).
\]
If 
\[
\mathrel \Big|\frac{r}{f^2}- \frac{1}{G_Q}\mathrel \Big|_\infty\leq \frac{1}{q^{3Q}},
\]
then $|\alpha_g|_\infty\leq q^{-2}$ if $g\le N$ and hence 
\begin{align}
\mathrel\Big|\sum\limits_{\substack{g \in \BF_q[t]\\ \deg g\le N}} e(\alpha_g) \mathrel\Big| = \# \{g \in \BF_q[t]: \deg g\le N \}
\end{align}
by definition of $e(x)$. Now for $\CS(Q)$ as defined in \eqref{FF2}, it follows that 
\begin{align*}
& \sum\limits_{\substack{f \in \BF_q[t] \\ \deg f \leq Q\\ f\, \mbox{\scriptsize \rm monic}}}\sum\limits_{\substack{r \md  f^2\\  (r,f)=1}}
\mathrel\bigg|\sum\limits_{\substack{g \in \BF_q[t]\\ \deg g\le N}} a_g  e \Big(g\cdot\frac{r}{f^2}\Big) \mathrel\bigg|^2 \\ 
& \geq \sum\limits_{ (r,f)\in \CS(Q)}\mathrel\bigg|\sum\limits_{\substack{g \in \BF_q[t]\\ \deg g\le N}}a_g e \Big(g\cdot\frac{r}{f^2}\Big) \mathrel\bigg|^2\\
& = \#\CS(Q)\cdot  {\# \{g \in \BF_q[t] : \deg g \le N \}}^{2}\\
& = \#\CS(Q)\cdot \# \{g \in \BF_q[t] : \deg g\le N\} \sum\limits_{\substack{g \in \BF_q[t]\\ \deg g\le N}}|a_g|^2\\
& = \#\CS(Q)\cdot q^{N+1} \sum\limits_{\substack{g \in \BF_q[t]\\ \deg g\le N}}|a_g|^2\\
& =\frac{1}{2}\cdot \#\CS(Q)\left(q^{N+1}+q^{3Q-1}\right)\sum\limits_{\substack{g \in \BF_q[t]\\ \deg g\le N}}|a_g|^2.\\
\end{align*}

\section{Main result} 
In the rest of this paper, we establish our main result, the following corrected version of \cite[Corollary 6.5.]{BaSi} 
for the case $k=2$ and char$(\mathbb{F}_q)>2$. 

\begin{Theorem} \label{themainresult}
Let $N,Q\in \mathbb{N}$ and $a_g\in \mathbb{C}$, where $g\in \BF[t]$. Assume that $2Q\le N\le 4Q$. Then
\begin{align} \label{mainineq}
& \sum\limits_{\substack{f \in \BF_q[t] \\ \deg f \leq Q}} \sum\limits_{\substack{r \bmod f^2\\ (r,f)=1}}\mathrel \bigg 
|\sum\limits_{\substack{g \in \BF_q[t]\\ \deg g\le N}} a_g 
e\Big(g\cdot \frac{r}{f^2}\Big)\bigg |^2\\
\ll_{q} & 2^{Q+N}\left(q^{3Q}+ \min\left\{q^{2Q+N/2},q^{Q/2+N}\right\}\right)\cdot 
 \sum\limits_{\substack{g \in \BF_q[t]\\ \deg g\le N}} |a_g|^2.
\end{align}
\end{Theorem}

This corresponds to the result for the classical case in \cite{BaZh} which asserts that  
\begin{equation} \label{squarenorm}
\begin{split}
& \sum\limits_{q\le Q} \sum\limits_{\substack{a \bmod{q^2}\\ (a,q)=1}} \left|\sum\limits_{n \le N}  
a_n \cdot e\left(\frac{na}{q^2}\right)\right|^2 \\
\ll & (QN)^{\varepsilon}\left(Q^3+\min\{Q^2\sqrt{N},\sqrt{Q}N\}\right)\sum\limits_{n\le N} |a_n|^2.
\end{split}
\end{equation}

{\bf Remark:} We note that the term $2^{Q+N}$ on the right-hand side of \eqref{mainineq} satisfies 
$$
2^{Q+N}\le q^{\varepsilon(Q+N)}
$$
as $q$ is large enough. In this sense, this term corresponds to the term $(QN)^{\varepsilon}$ on the right-hand side of \eqref{squarenorm}.
We further note that the $N$-range $2Q\le N\le 4Q$ in Theorem \ref{themainresult} corresponds to the most relevant
range $Q^2\le N\le Q^4$ in the classical case, and one can show by simple arguments 
that the claimed inequality \eqref{mainineq} remains true if $N$ lies outside this range (see \cite[inequality (1.3)]{BaZh}).\\

{\bf Acknowledgements:} The authors would like to thank the Ramakrishna Mission Vivekananda Educational and Research Institute for an excellent working environment and the anonymous referee for his careful checking of our paper. 

\section{Preliminaries}
In this section, we state some basic results needed for the rest of this paper. 
The following is the one-dimensional version of a general large sieve bound which can be 
found in \cite[section 4]{BaSi}. 

\begin{Lemma}\label{LS1}
Let $X_1, X_2,...,X_R \in \BF_q(t)_\infty$. Suppose that $0<\Delta\leq 1/q$ and $R\in \BN$. Set
\[
K(\Delta): = \max_{x\in \BF_q(t)_\infty }\sum\limits_{\substack{r=1\\ || X_r -x||\leq \Delta}}^{R}1.
\]
Then 
\[
\sum\limits_{r=1}^{R}\mathrel\big| S(X_r)\mathrel\big|^2 \ll K(\Delta)(q^{N+1} + \Delta^{-1})Z,
\]
with an absolute $\ll$-constant.
\end{Lemma}

This implies the following result. 

\begin{Lemma}\label{LS2}
Let $X_1, X_2,...,X_R \in \BF_q(t)_\infty$ and  $Y_1,Y_2,...,Y_L \in \BF_q(t)_\infty$, where $R,L\in \mathbb{N}$. Suppose that $0<\Delta\leq 1/q$ 
and for every  $x\in \BF_q(t)_\infty$, there exists $Y_l$ with $1\leq l \leq L$ such that 
\[
\| Y_l -x\|\leq \Delta.
\]
Put
\[
K'(\Delta): = \max_{1\leq l \leq L }\sum\limits_{\substack{r=1\\ || X_r -Y_l||\leq \Delta}}^{R}1.
\]
Then 
\[
\sum\limits_{r=1}^{R}\mathrel\big| S(X_r)\mathrel\big|^2 \ll K'(\Delta)(q^{N+1} + \Delta^{-1})Z,
\]
with an absolute $\ll$-constant.
\end{Lemma}

Further, we need the Poisson summation formula for function fields (see \cite[Theorem 4.2.1]{CaFr}). 

\begin{Lemma}[Poisson Summation Formula] \label{Poisson} Let $\Phi: \BF_q(t)_\infty \to \BC$ be a function such that 
$$
F(x)=\sum\limits_{f\in \BF_q[t]} |\Phi(f+x)|
$$ 
is uniformly convergent in compact subsets of $\BF_q(t)_\infty$ and 
$$
\sum\limits_{g\in \BF_q[t]} |\hat{\Phi}(g)|
$$ 
is convergent, where $\hat{\Phi}$ is the Fourier transform of $\Phi$, defined as
$$
\hat{\Phi}(x):=\int\limits_{\BF_q(t)_\infty} \Phi(y)e(-xy) \ {\rm d}y.
$$
Then 
\begin{align*}
\sum\limits_{f\in \BF_q[t]}\Phi(f) = \sum\limits_{g\in \BF_q[t]}\hat{\Phi}(g).
\end{align*}  
\end{Lemma}

We shall work with the weight function
\begin{equation} \label{Phi1def}
\Phi_1(x):=
\begin{cases}
1, &  \text{if} \ |x|_\infty \leq \frac{1}{q} \\
0, & \text{otherwise}.
\end{cases}
\end{equation}
For this function, the following holds by \cite[Lemma 2.2]{Hsu}. 

\begin{Lemma} \label{Phi1} We have $\Phi_1=\hat{\Phi}_1$.
\end{Lemma}

Now we quote  Dirichlet's approximation theorem for function fields of dimension $n\in \BN$ from \cite[Theorem 1.1]{GaGh}. 
This is an analogue of Dirichlet's theorem for local fields of positive characteristic (for more details see \cite{BM}).

\begin{Theorem}\emph{(Theorem 1.1.  in \cite{GaGh})}\label{DFF} Let $l$ be a nonnegative integer. For $x:= (x_1,x_2,\cdots,x_n)\in \BF_q(t)_\infty^n$, there exists $v =(v_1,v_2,\cdots,v_n) \in \BF_q[t]^n$ and $u \in \BF_q[t]\backslash \{0\}$ with $(u,v_i)=1$
for every $i,\ 1\leq i\leq n$ such that 
$$
|v_1x_1+v_2x_2+\cdots +v_nx_n-u|_{\infty}< \frac{1}{q^{nl}}\, \text{and} \, \max_{1\leq j\leq n} |v_j|_{\infty}\leq q^l.
$$
\end{Theorem}

We shall only need the one-dimensional case of the above theorem. 

\section{Quadratic Gauss sums} 
Recall that $\cha \BF_q \neq 2$. As in the classical case, we define the quadratic Gauss sums for the rational function field as 
\begin{equation}\label{Gsum}
G(\alpha, l; \beta) = \sum\limits_{d \bmod{\beta}} e\left(\frac{\alpha d^2 +ld}{\beta}\right). 
\end{equation}
These Gauss sums will play an important part in this paper. Along similar lines as in the classical setting, we will evaluate them in this section. 
Our first result is the following multiplicative property.

\begin{Lemma} \label{Gsum2} If $(\beta_1, \beta_2)=1$, then 
\[
G(\alpha, l; \beta_1\beta_2) =G(\alpha \beta_2, l; \beta_1)G(\alpha \beta_1, l; \beta_2).
\]
\end{Lemma}

\begin{proof}
We have
\begin{equation*}
\begin{split}
G(\alpha, l; \beta_1\beta_2) = & \sum\limits_{d \bmod{\beta_1\beta_2}} e\left(\frac{\alpha d^2 +ld}{\beta_1\beta_2}\right) \\
= & \sum\limits_{d_1 \bmod{\beta_1}} \sum\limits_{d_2 \bmod{\beta_2}} e\left(\frac{\alpha (d_1\beta_2+d_2\beta_1)^2+l(d_1\beta_2+d_2\beta_1)}{\beta_1\beta_2}\right)\\
= & \sum\limits_{d_1 \bmod{\beta_1}} e\left(\frac{\alpha \beta_2d_1^2 +ld_1}{\beta_1}\right) 
\sum\limits_{d_2 \bmod{\beta_2}} e\left(\frac{\alpha \beta_1d_2^2 +ld_2}{\beta_2}\right)\\
= & G(\alpha \beta_2, l; \beta_1)G(\alpha \beta_1, l; \beta_2),
\end{split}
\end{equation*}
which completes the proof.
\end{proof}

Further, we relate $G(\alpha, l; \beta)$ to $G(\alpha, 0; \beta)$. 

\begin{Lemma}\label{Gsum1} Assuming that $(\alpha, \beta) = 1$ and $\alpha \ol \alpha =1 \bmod \beta,$  we have 
\[
G(\alpha, l; \beta) = e\left(-\frac{\ol{\alpha} l^2}{4\beta}\right)G(\alpha, 0; \beta).
\]
\end{Lemma}

\begin{proof} Using quadratic completion, we obtain
$$
e\left(\frac{\alpha d^2 +ld}{\beta}\right)=
e\left(\frac{\alpha (d +2^{-1}l\overline{\alpha})^2}{\beta}\right)e\left(-\frac{\overline{\alpha}l^2}{4\beta}\right).
$$
Making the change of variable $d +2^{-1}l\overline{\alpha} \rightarrow d$ and summing over $d$ gives the desired equation.
\end{proof}

Next, we reduce the exponent in power moduli as follows.

\begin{Lemma} \label{Gsum3} Assuming $(\alpha, \beta) = 1$ and $r \geq 2$, we have
$$
G(\alpha, 0; \beta^r) = q^{\deg \beta}G(\alpha, 0; \beta^{r-2}).
$$
\end{Lemma}

\begin{proof} We write
\begin{equation*}
\begin{split}
 G(\alpha, 0; \beta^r) = & \sum\limits_{d \bmod{\beta^r}} e\left(\frac{\alpha d^2}{\beta^r}\right)\\ = & 
 \sum\limits_{a \bmod{\beta}} \sum\limits_{b \bmod{\beta^{r-1}}} e\left(\frac{\alpha (a\beta^{r-1}+b)^2}{\beta^r}\right)\\
 = & \sum\limits_{b \bmod{\beta^{r-1}}} e\left(\frac{\alpha b^2}{\beta^r}\right) \sum\limits_{a \bmod{\beta}} \
 e\left(\frac{2\alpha ba}{\beta}\right). 
\end{split}
\end{equation*}
Now 
\begin{equation*}
\begin{split}
\sum\limits_{a \bmod{\beta}} e\left(\frac{2\alpha ba}{\beta}\right) =  
\begin{cases}
q^{\deg \beta} & \mbox{ if } b \equiv 0 \bmod{\beta}\\ 0 & \mbox{ otherwise.} 
\end{cases}
\end{split}
\end{equation*}
It follows that
\begin{equation*}
\begin{split}
 G(\alpha, 0; \beta^r)
 = & q^{\deg \beta}\sum\limits_{\substack{b \bmod{\beta^{r-1}}\\ b \equiv 0 \bmod{\beta}}}
 e\left(\frac{\alpha b^2}{\beta^r}\right)\\
 = & q^{\deg \beta}\sum\limits_{d \bmod{\beta^{r-2}}} e\left(\frac{\alpha (\beta d)^2}{\beta^{r}}\right)\\ = 
 & q^{\deg \beta}G(\alpha, 0; \beta^{r-2}).
 \end{split}
\end{equation*}
\end{proof}

The next lemma reduces $G(\alpha,0;P)$ to $G(1,0;P)$ in the case when $P$ is an irreducible polynomial. 

\begin{Lemma} \label{Gsum4} If $P\in \BF_q[t]$ is an irreducible polynomial and $(\alpha,P)=1$, then   
\[
G(\alpha,0; P) =\left(\frac{\alpha}{P}\right) G(1, 0; P),
\]
where $\DS\left(\frac{\alpha}{P}\right)$ is the Legendre symbol for the rational function field.
\end{Lemma}

\begin{proof}
We first write
$$
G(\alpha,0;P)=1+\sum\limits_{\substack{e \bmod{P}\\ e\not\equiv 0 \bmod{P}}} \left(1+\left(\frac{e}{P}\right)\right)
e\left(\frac{\alpha e}{P}\right),
$$
which implies
\begin{equation} \label{gaussrew}
G(\alpha,0;P)= \sum\limits_{e \bmod{P}} \left(\frac{e}{P}\right)e\left(\frac{\alpha e}{P}\right)
\end{equation}
using the orthogonality relation
$$
\sum\limits_{e \bmod{P}} e\left(\frac{\alpha e}{P}\right)=0
$$
if $(\alpha,P)=1$. 
Changing variables $f=\alpha e$ now gives
$$
G(\alpha,0;P)=\sum\limits_{e \bmod{P}} \left(\frac{\overline{\alpha} f}{P}\right) e\left(\frac{f}{P}\right) =
\left(\frac{\alpha}{P}\right)\sum\limits_{f \bmod{P}} \left(\frac{f}{P}\right) e\left(\frac{f}{P}\right),
$$
where $\alpha \overline{\alpha}\equiv 1 \bmod{P}$. 
Similarly as above,
$$
\sum\limits_{f \bmod{P}} \left(\frac{f}{P}\right) e\left(\frac{f}{P}\right) = G(1,0;P)
$$
which completes the proof.
\end{proof}

Now we are ready to determine the modulus of a quadratic Gauss sum. 

\begin{Lemma} \label{Gsum5} Assuming $(\alpha,\beta)=1$, we have 
\begin{equation}\label{Gsum6}
|G(\alpha, l; \beta)| = |\beta|^{1/2}_\infty.
\end{equation}
\end{Lemma}

\begin{proof}
By the virtue of the previous lemmas on quadratic Gauss sums, it suffices to show that 
\begin{equation} \label{goaly}
|G(1, 0; P)| = |P|^{1/2}_\infty
\end{equation}
for any irreducible polynomial $P$, which we shall establish in the following. Taking the modulus square of both sides of \eqref{gaussrew} 
gives 
\begin{equation} \label{modsquare}
|G(\alpha,0; P)|^2 =  \sum\limits_{f_1, f_2 \bmod P} \left(\frac{f_1}{P}\right)\left(\frac{f_2}{P}\right)
e\left(\frac{ (f_1-f_2)\alpha}{P}\right)
\end{equation}
for any $\alpha$ with $\alpha\not\equiv 0 \bmod{P}$. 
We observe that the right-hand side of \eqref{modsquare} equals 0 if $\alpha \equiv 0 \bmod{P}$ using the orthogonality relation
$$
\sum\limits_{f \bmod P} \left(\frac{f}{P}\right)=0.
$$
Now summing both sides of \eqref{modsquare} over all $\alpha \bmod P$ with 
$\alpha \not\equiv 0 \bmod{P}$ and then using Lemma \ref{Gsum4} and the above observation, we obtain 
\begin{equation}
\begin{split}
\left(q^{\deg P}-1\right) |G(1,0; P)|^2 = & \sum\limits_{\alpha \bmod{P}} \sum\limits_{f_1, f_2 \bmod P} \left(\frac{f_1}{P}\right)\left(\frac{f_2}{P}\right)
 e\left(\frac{ (f_1-f_2)\alpha}{P}\right)\\ 
 = & \sum\limits_{f_1, f_2 \bmod P} \left(\frac{f_1}{P}\right)\left(\frac{f_2}{P}\right)
 \sum\limits_{\alpha \bmod{P}} e\left(\frac{ (f_1-f_2)\alpha}{P}\right)\\
 = & q^{\deg P} \sum\limits_{f\bmod P} \left(\frac{f}{P}\right)^2\\
 = & q^{\deg P} \left(q^{\deg P}-1\right),
\end{split}
\end{equation}
which gives us \eqref{goaly}. 
\end{proof}

\section{Quadratic exponential integrals}\label{QEIntegrals}
Let $\BF_q(t)_{\infty}^2$ be the set of squares of elements of $\BF_q(t)_{\infty}$. Observe that $y\in \BF_q(t)_{\infty}^2\setminus \{0\}$ if and only if the leading coefficient $c_R$ of
$$
y=\sum\limits_{i=-\infty}^R c_it^i\quad (c_R\not=0)
$$
is a square in $\BF_q^{\ast}$ and $R$ is even. Now we fix a square root function $\sqrt{y}$ on $\BF_q(t)_{\infty}^2$ as follows. If $y=0$, then $\sqrt{y}=0$. If $y$ is a square in $\BF_q^{\ast}$, 
then $\sqrt{y}$ is any of the two $s\in \BF_q$ such that $s^2=y$. Now, more generally, if $y\in \BF_q(t)_{\infty}^2\setminus \{0\}$, then $\sqrt{y}$ is those of the two 
$s\in \BF_q(t)_{\infty}^{\ast}$ satisfying $s^2=y$ whose leading coefficients is the square root  
$\sqrt{c_R}$, fixed above, of the leading coefficient $c_R$ of $y$. 

Now let $Q$ be a positive integer and $\mathcal{B}^2(0, 2Q)$ be the set of squares of elements of $\BF_q(t)_{\infty}$ in the ball 
$\mathcal{B}(0, 2Q)$.
In this section, we evaluate exponential integrals of the form 
\begin{equation} \label{expindef}
E(A,B) : = \int\limits_{\mathcal{B}^2(0, 2Q)} \frac{1}{2|\sqrt{y}|_{\infty}} e\left(Ay-  B\sqrt{y}\right) dy,
\end{equation}
which will show up in this paper as well.

A change of variables $y=x^2$ gives
\begin{equation} \label{changeofvariables}
E(A,B) = \int\limits_{\mathcal{B}(0, Q)} e\left(Ax^2-  Bx\right) dx,
\end{equation}
where we note that 
$$
dx^2=2|x|_{\infty}dx.
$$ 
If $A=0$, we immediately deduce the following.

\begin{Lemma} \label{expinte1} For every $B\in \BF_q(t)_{\infty}$, we have 
\begin{equation*} 
E(0,B)= \begin{cases}q^{Q +1} \,\hspace{.2cm} \text{if } \, |B|_\infty \leq q^{-(Q +2)}\\
0 \, \ \hspace{1cm} \text{otherwise}.
\end{cases}
\end{equation*}
\end{Lemma} 

If $A\not=0$, then we proceed as follows. First, using quadratic completion, we obtain
\begin{align*}
E(A,B) = e\left(\frac{-B^2}{4A}\right)\int\limits_{\mathcal{B}(0, Q)} e\left(A \left(x-\frac{B}{2A}\right)^2 \right) dx.
\end{align*}
If $|B/A|_\infty\leq q^{Q}$, then a linear change of variables gives 
\begin{align*}
\int\limits_{\mathcal{B}(0, Q)} e\left(A \left(x-\frac{B}{2A}\right)^2 \right) dx = 
& \int\limits_{\mathcal{B}(0, Q)} e\left(A x^2 \right) dx,
\end{align*}
where we use the fact that $\mathcal{B}(B/(2A),Q)=\mathcal{B}(0,Q)$ in this case. 
When $|B/A|_\infty > q^{Q}$, we get 
\begin{align*}
\int\limits_{\mathcal{B}(0, Q)} e\left(A \left(x-\frac{B}{2A}\right)^2 \right) dx = & 
\int\limits_{\mathcal{B}(B/(2A), Q)} e\left(A x^2 \right) dx. 
\end{align*}  
Now observe that it is always possible to write $A=\alpha A_2$, where $A_2\in \BF_q(t)_{\infty}^2\setminus\{0\}$ and 
$\alpha=ct^{-\epsilon}$ with $c\in \BF_q^{\ast}$ and $\epsilon\in \{0,1\}$
suitable. 
Then making the change of variables $y=\sqrt{A_2}x$, we obtain
\begin{equation*}
\begin{split}
\int\limits_{\mathcal{B}(0, Q)} e\left(A x^2 \right) dx = & |\sqrt{A_2}|_\infty^{-1}
\int\limits_{\mathcal{B}(0, Q+\deg \sqrt{A_2})} e\left( \alpha y^2 \right) dy\\ 
= & q^{-\lceil(\deg A)/2\rceil}
\int\limits_{\mathcal{B}(0, Q+\lceil(\deg A)/2\rceil)} e\left( \alpha y^2 \right) dy
\end{split}
\end{equation*}
and 
\begin{equation*}
\begin{split}
\int\limits_{\mathcal{B}(B/(2A), Q)} e\left(A x^2 \right) dx = & |\sqrt{A_2}|_\infty^{-1}
\int\limits_{\mathcal{B}(B\sqrt{A_2}/(2A), Q+ \deg \sqrt{A_2})} e\left( \alpha y^2 \right) dy\\ = & q^{-\lceil(\deg A)/2\rceil}
\int\limits_{\mathcal{B}(B\sqrt{A_2}/(2A), Q+\lceil (\deg A)/2\rceil)} e\left( \alpha y^2 \right) dy,
\end{split}
\end{equation*}
where we note that
$$
dCx=|C|_{\infty}dx.
$$

Summarizing the above, we have the following.

\begin{Lemma} \label{expinte2} Suppose that $A=\alpha A_2$ with $A_2\in \BF_q(t)_{\infty}^2\setminus\{0\}$ and $\alpha=ct^{-\epsilon}$ for some $c\in \BF_q^{\ast}$ and $\epsilon\in \{0,1\}$. Suppose further that
$B\in \BF_q(t)_{\infty}$. Then
$$
E(A,B) = q^{-\lceil(\deg A)/2\rceil}\cdot e\left(\frac{-B^2}{4A}\right)\cdot \int\limits_{\mathcal{B}(C, Q+
\lceil (\deg A)/2\rceil )} e\left( \alpha y^2 \right) dy,
$$
where 
$$
C:=\begin{cases} 0 & \mbox{ if } |B/A|_\infty\leq q^{Q} \\ B\sqrt{A_2}/(2A) & \mbox{ if } |B/A|_\infty> q^{Q}. \end{cases}
$$
\end{Lemma} 

It remains to evaluate integrals of the form $\int\limits_{\mathcal{B}(x, n)} e( \alpha y^2 ) dy$, which is done in the following lemma.

\begin{Lemma} \label{expinte3} Suppose that $\alpha=ct^{-\epsilon}$ for some $c\in \BF_q^{\ast}$ and $\epsilon\in \{0,1\}$. Suppose further that
$x\in \BF_q(t)_{\infty}$ and $n \in \mathbb{Z}$. Then 
\begin{equation} \label{infinalev}
\int\limits_{\mathcal{B}(x, n)} e(\alpha y^2)dy=\begin{cases} q^{n+1} & \mbox{ if } \deg x\le n\le -1\\
1+\epsilon(s(c)q^{1/2}-1) & \mbox{ if } \deg x\le n \mbox{ and } n\ge 0\\
e(\alpha x^2)\cdot q^{n+1} & \mbox{ if } \min\{\deg x,-\deg x-1+\epsilon\}>n
\\ 0 & \mbox{ if } \deg x> n\ge -\deg x-1+\epsilon, \end{cases}
\end{equation}
where 
\begin{equation} \label{scdef}
s(c)= \begin{cases}
+1 & \mbox{ if } c \mbox{ is a square in } \BF_q^{\ast}\\ -1 & \mbox{ otherwise.}  
\end{cases}
\end{equation}
\end{Lemma}

\begin{proof}
First, just using the definition of $e(\cdots)$, we observe that 
\begin{equation} \label{particular}
\int\limits_{\mathcal{B}(0,n)} e(\alpha y^2) dy= \mu(\mathcal{B}(0,n))= q^{n+1}
\end{equation}
if $n\le -1$.  

Now suppose that $\deg x\le n$. Then
$$
\int\limits_{\mathcal{B}(x, n)} e(\alpha y^2 ) dy= \int\limits_{\mathcal{B}(0, n)} e(\alpha y^2 ) dy.
$$
If $n\le -1$, it follows that
$$
\int\limits_{\mathcal{B}(x, n)} e(\alpha y^2 ) dy=q^{n+1}
$$
using \eqref{particular}. If $n\ge 0$, then we get
$$
\int\limits_{\mathcal{B}(x, n)} e( \alpha y^2 ) dy= \int\limits_{\mathcal{B}(0, n)} e(\alpha y^2 ) dy = 
\int\limits_{\mathcal{B}(0, n)\setminus \mathcal{B}(0,-1)} e(\alpha y^2 ) dy+1,
$$
again using \eqref{particular}. Writing 
\begin{equation} \label{yexp}
y=\sum\limits_{i=-\infty}^R c_it^i \quad \mbox{ with } R\ge 0, \ c_R\not=0,
\end{equation}
we further have 
\begin{equation*}
\begin{split}
& \int\limits_{\mathcal{B}(0, n)\setminus \mathcal{B}(0,-1)} e(\alpha y^2 ) dy\\ = & \sum\limits_{R=0}^{n}\
\int\limits_{\deg y=R} E\left( \TS c\sum\limits_{i =0}^{R}z_ic_ic_{-i-1+\epsilon}\right)dy\\
= & \sum\limits_{R=0}^{n}\sum\limits_{k\in \BF_q}E(k)\mu\Big(\Big\{\deg y=R \ :\ c\sum\limits_{i =0}^{R}z_ic_ic_{-i-1+\epsilon} =k \Big\}\Big)\\
= & \sum\limits_{R=0}^{n}\sum\limits_{k\in \BF_q}E(k)\mu\Big(\Big\{\deg y=R \ :\ c_{-R-1+\epsilon}=(z_Rcc_R)^{-1}
\Big(k-c\sum\limits_{i =0}^{R-1}z_ic_ic_{-i-1}\Big)\Big\}\Big),
\end{split}
\end{equation*}
where 
\begin{equation} \label{zi}
z_i:=\begin{cases} 2 & \mbox{ if } (i,\epsilon)\not=(0,1) \\ 1 & \mbox{ if } (i,\epsilon)=(0,1). \end{cases}
\end{equation}
If $(R,\epsilon)\not=(0,1)$, then the measure in the last line is independent of $k$, namely
\begin{equation*}
\mu\Big(\Big\{\deg y=R \ :\  c_{-R-1+\epsilon}=(z_Rcc_R)^{-1}
\Big(k-c\sum\limits_{i =0}^{R-1}z_ic_ic_{-i-1}\Big)\Big\}\Big)
= q^R\left(1-\frac{1}{q}\right).
\end{equation*}
From the orthogonality relation 
\begin{equation} \label{ortho}
\sum\limits_{k\in \BF_q}E(k)=0,
\end{equation}
it then follows that the contribution of $(R,\epsilon)\not=(0,1)$ equals 0.
If $(R,\epsilon)=(0,1)$, then
\begin{equation*}
\begin{split}
& \sum\limits_{k\in \BF_q}E(k)\mu\Big(\Big\{\deg y=R \ :\ c_{-R-1+\epsilon}=(z_Rcc_R)^{-1}
\Big(k-c\sum\limits_{i =0}^{R-1}z_ic_ic_{-i-1}\Big)\Big\}\Big)\\
= & \sum\limits_{k\in \BF_q}E(k)\mu\Big(\Big\{\deg y=0 \ :\ c_{0}=(cc_0)^{-1}k\Big\}\Big)
= \sum\limits_{c_0\in \BF_q^{\ast}} E(cc_0^2).
\end{split}
\end{equation*}
The latter is a classical quadratic Gauss sum and has the value
$$
\sum\limits_{c_0\in \BF_q^{\ast}} E(cc_0^2)= \sum\limits_{c_0\in \BF_q} E(cc_0^2) -1= s(c)q^{1/2}-1,
$$
where $s(c)$ is defined as in \eqref{scdef}.
Hence, if $\deg x\le n$ and $n\ge 0$, then we obtain
\begin{equation*}
\int\limits_{\mathcal{B}(x, n)} e(\alpha y^2)dy =1+\epsilon(s(c)q^{1/2}-1). 
\end{equation*}

Now suppose that $\deg x>n$. If $\min\{\deg x,-\deg x-1+\epsilon\}>n$, then, similarly as in \eqref{particular},
$$
\int\limits_{\mathcal{B}(x,n)} e(\alpha y^2) dy = e(\alpha x^2)\cdot \mu((\mathcal{B}(x,n))=e(\alpha x^2)\cdot q^{n+1}. 
$$

Finally, we consider the case when $R:=\deg x> n\ge -\deg x-1+\epsilon$. In this case, $y\in \mathcal{B}(x,n)$ implies $\deg y=R$. Using the 
same notations as in \eqref{yexp} and \eqref{zi}, we get
\begin{equation*}
\begin{split}
& \int\limits_{\mathcal{B}(x,n)} e(\alpha y^2)dy\\ = & 
\int\limits_{\substack{y\in \BF_q(t)_{\infty}\\ \deg(y-x)\le n}} E\left( c\sum\limits_{i =0}^{R} z_ic_ic_{-i-1+\epsilon}\right)dy\\
= & \sum\limits_{k\in \BF_q}E(k)\mu\Big(\Big\{y\in \BF_q(t)_{\infty}\ :\ \deg(y-x)\le n,\ c\sum\limits_{i =0}^{R}z_ic_ic_{-i-1+\epsilon} =k
\Big\}\Big)\\
= & \sum\limits_{k\in \BF_q}E(k)\mu\Big(\Big\{y\in \BF_q(t)_{\infty}\ :\ \deg(y-x)\le n,\\ & \hspace*{5.0cm} c_{-R-1+\epsilon}=(z_Rcc_R)^{-1}
\Big(k-c\sum\limits_{i =0}^{R-1}z_ic_ic_{-i-1+\epsilon}\Big)\Big\}\Big).
\end{split}
\end{equation*}
We observe that the case $(R,\epsilon)=(0,1)$ does not occur here because of our condition $R> n\ge -R-1+\epsilon$, and
therefore  
the measure in the last line is always independent of $k$, namely
\begin{equation*}
\begin{split}
& \mu\Big(\Big\{y\in \BF_q(t)_{\infty}\ :\ \deg(y-x)\le n,\\ & \hspace{3.3cm} c_{-R-1+\epsilon}=(z_Rcc_R)^{-1}
\Big(k-c\sum\limits_{i =0}^{R-1}z_ic_ic_{-i-1+\epsilon}\Big)\Big\}\Big) 
= q^{n}.  
\end{split}
\end{equation*}
Again, from the orthogonality relation \eqref{ortho}, it then follows that
$$
\int\limits_{\mathcal{B}(x, n)} e( y^2 )=0.
$$
Combining everything, we obtain \eqref{infinalev}. \end{proof}

\section{Diophantine approximation}\label{DA}
After having provided the basic tools used in this paper, we are ready to investigate the large sieve with square moduli. We aim to estimate 
the quantity 
\begin{align*}
\sum\limits_{\substack{f \in \BF_q[t]\\ \deg f=Q}} \ \sum\limits_{\substack{r \bmod f^2,\\ (r,f)=1}}\mathrel \bigg 
|\sum\limits_{\substack{g \in \BF_q[t]\\ \deg g\le N}} a_g  e\Big(g\cdot \frac{r}{f^2}\Big)\bigg |^2.
\end{align*}
To this end, we use Lemma \ref{LS2}. In our situation, we let $X_1, \cdots, X_R$ be the sequence of Farey fractions $r/f^2$  with $\deg f = Q$, $\deg r \leq 2\deg f-1$
and $(r,f)=1$ so that the above expression equals
\begin{align*}
\sum\limits_{r=1}^{R}\mathrel \bigg |\sum\limits_{\substack{g \in \BF_q[t]\\ \deg g\le N}} a_g  e\Big(g\cdot X_r\Big)\bigg |^2.
\end{align*}
The $Y_l$'s are now chosen as follows. First we set
\[
\tau: = \frac{1}{\sqrt \Delta}
\]
We let $Y_1,Y_2,...,Y_L$ be the points  
\[
\frac{u}{v}+\frac{1}{f_kv^2},
\]
where 
$$
v\in \BF_q[t]\setminus \{0\}, \ |v|_\infty \leq  \tau, \ (u,v)=1, \ \deg u< \deg v,
$$ 
and  the $f_k$'s are polynomials of degree $k \in \BN_{}$ with 
\begin{equation} \label{kcond}
K:=\lceil \log_q \tau-\deg v\rceil \leq k\leq \kappa:=\lceil2\log_q \tau-2\deg v\rceil.
\end{equation}
We want to show that the $ Y_l$'s above satisfy the conditions in Lemma \ref{LS2} if the $f_k$'s are chosen suitably.

By Dirichlet's approximation theorem for function fields, Theorem \ref{DFF} with dimension $n=1$, every $x\in \BF_q(t)_\infty$ can be written in the form 
\begin{align}\label{cond1}
x = \frac{u}{v} + z,\hspace{.6cm} \text{where}\, \, |v|_\infty \leq  \tau, \, \, (u,v)=1, \,\, |vz|_\infty\leq \frac{1}{\tau}.
\end{align}
We must show that for every $z$ with $|vz|_\infty\leq {\tau}^{-1}$, there exists $k\in \BN_{}$ satisfying \eqref{kcond} such that 
\[
\mathrel \Big |z-\frac{1}{f_kv^2}\mathrel\Big|_\infty \leq \Delta.
\]
First, if $\deg z \leq -(2 \deg v +\kappa)$, then 
\[
\mathrel \Big |z-\frac{1}{f_\kappa v^2}\mathrel\Big|_\infty \leq  \max\mathrel\Big\{|z|_\infty, \mathrel \Big|\frac{1}{f_\kappa v^{2}}\mathrel\Big|_\infty \mathrel \Big\}\leq q^{-\kappa -2\deg v}\leq q^{-2\log_q \tau} = \Delta.
\]
Otherwise, if  $\deg z > -(2 \deg v +\kappa) = -\lceil 2\log_q \tau\rceil$, then 
\[
\log_q \tau+\deg v +1 >  \deg f_K v^2\geq  \log_q \tau+\deg v, 
\]
i.e. 
\[
-\log_q \tau-\deg v -1 < \deg \frac{1}{f_K v^2} \leq  -\log_q \tau-\deg v.
\]
Since 
$$
-\lceil 2\log_q \tau\rceil<\deg z \leq  -\log_q \tau-\deg v,
$$
we can now choose  $f_k$ with  $k\geq K$ in such a way that the leading coefficient  of $z$ is cancelled by that of $1/(f_k v^2)$, so that 
\[
\deg \left(z-\frac{1}{f_kv^2}\right) \leq  -\lceil 2\log_q \tau\rceil,
\]
i.e.
\[
\mathrel \Big |z-\frac{1}{f_kv^2}\mathrel\Big|_\infty \leq \Delta. 
\]

For $x\in \BF_q(t)_\infty$ we put
\[
P(x): = \sum\limits_{\substack{\deg f = Q,\, (r,f)=1\\ |r/f^2-x|_\infty\leq \Delta}} 1.
\]
Then we have 
\[
K'(\Delta)\leq  \max_{1\leq l \leq L } P(Y_l).
\]
Summarizing the above observations, we deduce the following.

\begin{Lemma} \label{K-P} We have 
\[
K'(\Delta)\leq  \max_{\substack{v\in \BF_q[t]\setminus \{0\}\\ \deg v \leq \log_q \tau} }\,\max_{\substack{u\in \BF_q(t)\\ (u,v)=1}}\, \max_{K\leq k\leq \kappa}P\Big(\frac{u}{v}+ \frac{1}{f_kv^2}\Big).
\]
\end{Lemma}

By the preceding lemma, it suffices to estimate $P(x)$ for $x$ of the form
\begin{align}\label{cond2}
x = \frac{u}{v} + z,\hspace{.6cm} \text{where}\, \, |v|_\infty \leq  \tau, \, \, (u,v)=1, \,\, z= 1/f_kv^2,
\end{align}
where $f_k$ is a polynomial of degree $k \in \BN$ satisfying \eqref{kcond}. We note that $x$ satisfies \eqref{cond1} if it satisfies \eqref{cond2}.

\section{First estimate for $P(x)$}
In this section, we establish a first estimate for $P(x)$ by applying Poisson summation, a Weyl shift to reduce quadratic to linear exponential 
sums and a counting argument.
First we set 
\begin{align*}
\Phi(x): & = 
\begin{cases}
1, &  \text{if} \ |x|_\infty \leq 1 \\
0, & \text{otherwise}
\end{cases} \\
&=
\begin{cases}
1, &  \text{if} \ |\frac{x}{t}|_\infty \leq \frac{1}{q} \\
0, & \text{otherwise}.
\end{cases}
\end{align*}
Then $\Phi(x)= \Phi_1(x/t)$ with 
\[
\Phi_1(x)=
\begin{cases}
1, &  \text{if} \ |x|_\infty \leq \frac{1}{q} \\
0, & \text{otherwise}.
\end{cases}
\]
Now define  
\begin{equation} \label{deltadef}
\omega: = \lceil \log_q\Delta\rceil + 1.
\end{equation}
Then it follows that
\begin{equation*}
P(x) \leq \sum\limits_{\deg f = Q}\sum\limits_{r\in \BF_q[t]} \Phi_1\Big(\frac{r-x f^2}{f^2 t^\omega}\Big).
\end{equation*}

Applying the Poisson summation formula, Lemma \ref{Poisson},   
with a linear change of variable to the sum over $r$, and using $\Phi_1=\hat{\Phi}_1$ (see Lemma \ref{Phi1}), we deduce that 
\begin{equation*}
P(x)\le q^{2Q+\omega} \sum\limits_{\deg f = Q}\sum\limits_{r'\in \BF_q[t]} \Phi_1\big(f^2 t^\omega r'\big)e(x r' f^2),
\end{equation*}
which implies
\begin{equation} \label{Pxbound}
\begin{split}
P(x)\le & q^{2Q+\omega} \sum\limits_{\substack{r'\in \BF_q[t]\\ \deg r'\leq -2Q-\omega -1} }\sum\limits_{\deg f = Q}e(x r' f^2)\\
\leq & q^{3Q+\omega+1}+q^{2Q+\omega} \sum\limits_{\substack{r'\in \BF_q[t]\\ 0\le \deg r'\leq -2Q-\omega -1} }\mathrel\Bigg|\sum\limits_{\deg f = Q}e(x r' f^2)\mathrel\Bigg|,
\end{split}
\end{equation}
where the second line arises from isolating the contribution of $r'=0$. 
Applying the Cauchy-Schwarz inequality and 
$$
\sum\limits_{\substack{r'\in \BF_q[t]\\ 0\le \deg r'\leq -2Q-\omega -1} } 1 = q^{-2Q-\omega},
$$
we deduce from \eqref{Pxbound} that 
\begin{equation}\label{Eq3}
P(x)^2 \ll  q^{6Q+2\omega+2}+q^{2Q+\omega} \sum\limits_{\substack{r'\in \BF_q[t]\\ 0\le \deg r'\leq -2Q-\omega -1} }\mathrel\Bigg|\sum\limits_{\deg f = Q}e(x r' f^2)\mathrel\Bigg|^2.
\end{equation}

To bound the inner-most sum, we perform a Weyl shift.  Using a change of variables $h=f'-f$, we bound the modulus square by 
\begin{align*}
\mathrel\Bigg|\sum\limits_{\deg f = Q}e(x r' f^2)\mathrel\Bigg|^2 = & \sum\limits_{\substack{\deg f = Q \\ \deg f'= Q}} e\Big(x r' (f'-f)(f'+f)\Big) \\
 = & (q-1)q^{Q} + \sum\limits_{0\le \deg h \leq Q}\sum\limits_{\deg f = Q}e\Big(x r' h(h+2f)\Big)\\
\le  & q^{Q+1} + \sum\limits_{0\le \deg h \leq  Q}\mathrel\bigg|\sum\limits_{\deg f = Q}e\Big(2x r' hf\Big)\mathrel\bigg|.
\end{align*}
Combining this with \eqref{Eq3}, we deduce that 
\begin{align*}
P(x)^2 \ll_q  & q^{6Q+2\omega} + q^{3Q+\omega}+q^{2Q+\omega}\sum\limits_{\substack{r'\in \BF_q[t]\\ 0\le \deg r'\leq -2Q-\omega -1} } 
\sum\limits_{0\le \deg h \leq  Q}\mathrel\bigg|\sum\limits_{\deg f = Q}e\Big(2x r' hf\Big)\mathrel\bigg| \\
\le & q^{6Q+2\omega} + q^{3Q+\omega}+q^{2Q+\omega}
\sum\limits_{\substack{l\in \BF_q[t]\\ 0\le \deg l \leq -Q -\omega -1} }\tau(l) \mathrel\Big|\sum\limits_{\deg f = Q}e\Big(xlf\Big)\mathrel\Big|,
\end{align*}
where $\tau(l)$ is the number of divisors of $ l= 2r'h$  in $\BF_q[t]$. Here we recall our assumption that $q$ is not a power of 2. We note that 
\begin{equation} \label{tauest}  
\tau(l) \le 2^{\deg l}q.
\end{equation}
It follows that 
\begin{equation}\label{Eq3.1}
P(x)^2 \ll_q  q^{6Q+2\omega} + q^{3Q+\omega}+ 2^{-Q -\omega}q^{2Q+\omega}
\sum\limits_{\substack{l\in \BF_q[t]\\ 0\le \deg l \leq -Q -\omega -1} } \mathrel\Big|\sum\limits_{\deg f = Q}e\Big(xlf\Big)\mathrel\Big|.
\end{equation}

Now, using the Poisson summation formula, Lemma \ref{Poisson}, and $\Phi_1=\hat{\Phi}_1$ again, we have 
\begin{equation*}
\begin{split}
\sum\limits_{\deg f = Q}e\big(x  lf\big) =  & \sum\limits_{\deg f \le Q}e\big(x  lf\big) - \sum\limits_{\deg f \le Q-1}e\big(xlf\big)\\
= & \sum\limits_{ f \in \BF_q[t]}e\big(xlf\big) \Phi_1\big(t^{-Q-1}f\big)  - \sum\limits_{ f \in \BF_q[t]}e\big(xlf\big) \Phi_1\big(t^{-Q}f\big)\\
= & q^{Q+1}\sum\limits_{ f \in \BF_q[t]}\hat \Phi_1\Big(t^{Q+1}(f-xl)\Big) - q^{Q}\sum\limits_{ f \in \BF_q[t]}\hat \Phi_1\Big(t^{Q}(f-xl)\Big).
\end{split}
\end{equation*}
We observe that, for $n\in \mathbb{N}$,
$$
\sum\limits_{ f \in \BF_q[t]}\hat \Phi_1\Big(t^{n}(f-xl)\Big) = \begin{cases}
1, &  \, \text{if}\, \, \|xl\| \leq q^{-n-1} \\
0, & \, \text{otherwise},
\end{cases}
$$
with $||xl||$ as defined in \eqref{metric}. It follows that
\begin{equation} \label{Eq4}
\begin{split}
\left| \sum\limits_{\deg f = Q}e\big(xlf\big) \right| = & \begin{cases}
-q^{Q}, &  \, \text{if } \|xl\| = q^{-Q-1} \\
q^{Q+1}-q^Q &  \, \text{if } \|xl\|  \leq q^{-Q-2} \\
0, & \, \text{otherwise}
\end{cases}\\
\le & q^{Q+1}\mathcal{I}_{q^{-Q-1}}(\|xl\|),
\end{split}
\end{equation}
where 
$$
I_y(x):=\begin{cases} 1 & \mbox{ if } x\le y\\ 0 & \mbox{ if } x>y. \end{cases}
$$
Combining \eqref{Eq3.1} and \eqref{Eq4}, we have
\begin{equation}\label{Eq5}
P(x)^2\ll q^{6Q+2\omega} + q^{3Q+\omega}+ 2^{-Q -\omega}q^{3Q+\omega}
\sum\limits_{\substack{l\in \BF_q[t]\\ 0\le \deg l \leq -Q -\omega -1} } \mathcal{I}_{q^{-Q-1}}(\|xl\|).
\end{equation}
 
Recall that $x=u/v+z$. Writing $l = Av + k$ with unique $A,k\in \BF_q[t]$ such that
$\deg k <\deg v$, we transform the sum over $l$ in equation (\ref{Eq5}) into 
\begin{equation} \label{transform}
\begin{split}
& \sum\limits_{\substack{l\in \BF_q[t]\\ 0\le \deg l \leq -Q -\omega -1} } \mathcal{I}_{q^{-Q-1}}(\|xl\|) \\
\le & \sum\limits_{\substack{\deg A \le \\ -Q -\omega -1-\deg v}} ~~~~\sum\limits_{\deg k  \leq \deg v-1} \mathcal{I}_{q^{-Q-1}}
\mathrel \bigg(\mathrel\Big\|\mathrel\Big(\frac{u}{v} + z\mathrel\Big)(Av+k)\mathrel\Big\|\mathrel\bigg)\\
= & \sum\limits_{\substack{\deg A \le \\ -Q -\omega -1-\deg v}} ~~~~\sum\limits_{\deg k  \leq \deg v-1} \mathcal{I}_{q^{-Q-1}}\mathrel \bigg(\mathrel\Big\|
Avz +\frac{ku}{v}+ kz\mathrel\Big\|\mathrel\bigg).
\end{split}
\end{equation}
Now assume $k_1\not=k_2$ and $\deg k_1,\deg k_2  \leq \deg v-1$. Using the triangle inequality, we have
\begin{equation*}
\begin{split}
& \left| \left\{ Avz +\frac{k_1u}{v}+ k_2 \right\}- \left\{Avz +\frac{k_2u}{v}+ k_2z \right\} \right|_\infty \\
= & \mathrel\Big\| \frac{(k_1-k_2)u}{v}+ (k_1-k_2)z \mathrel\Big\|\\ \ge & 
\mathrel\Big\|\frac{(k_1-k_2)u}{v} \mathrel\Big\| - ||(k_1-k_2)z||.
\end{split}
\end{equation*}
Furthermore,
$$
\|(k_1-k_2)z\| < |vz|_\infty \leq \Delta^{1/2}\leq |v|^{-1}_\infty
$$
and hence 
$$
\|(k_1-k_2)z\|_\infty \le |v|^{-1}_\infty q^{-1}.
$$
It follows that
$$
\mathrel\Big\|\frac{(k_1-k_2)u}{v} \mathrel\Big\| - ||(k_1-k_2)z||_{\infty}
\ge |v|^{-1}_\infty-|v|^{-1}_\infty q^{-1} \ge
|v|^{-1}_\infty q^{-1}
$$
and hence
\begin{equation} \label{dist}
\mathrel\Bigg| \mathrel\Big\| Avz +\frac{k_1u}{v}+ k_2 \mathrel\Big\|- \mathrel\Big\|Avz +\frac{k_2u}{v}+ k_2z \mathrel\Big\| \mathrel\Bigg|_\infty \ge |v|^{-1}_\infty 
q^{-1}.
\end{equation}

The maximum number of points in $\BF_q(t)_\infty$ of mutual distance greater or equal $d=q^{-D}$ fitting into the ball $\{z\in \BF_q(t)_\infty\ :\ |z|_{\infty}\le q^{-E}\}$ 
is bounded by $1+q^{D-E+1}$. Hence, taking \eqref{dist} into account, we deduce that
$$
\sum\limits_{\deg k  \leq \deg v-1} \mathcal{I}_{q^{-Q-1}}\mathrel \bigg(\mathrel\Big\|
Avz +\frac{ku}{v}+ kz\mathrel\Big\|\mathrel\bigg) \le 1+\|v\|_\infty q^{-Q} = 1+q^{\deg v - Q}.
$$
Combining this with \eqref{transform}, we obtain 
\begin{equation*}
\sum\limits_{\substack{l\in \BF_q[t]\\ \deg l \leq -Q -\omega -1} } \mathcal{I}_{q^{-Q-1}}(\|xl\|) \ll  \left(1+q^{-Q -\omega -\deg v}\right)
\left(1+q^{\deg v-Q}\right),
\end{equation*}
which together with \eqref{Eq5} implies
\begin{equation*}
\begin{split}
P(x)^2\ll_q & q^{6Q+2\omega} + \left(1+2^{-Q -\omega}\right)q^{3Q+\omega}
\left(1+q^{-Q -\omega -\deg v}\right)
\left(1+q^{\deg v-Q}\right)\\ 
\ll_q & \left(q^{6Q}\Delta^2+\left(1+2^{\mathcal{L}-Q}\right)
\left(q^{3Q}\Delta+q^{2Q}|v|_{\infty}^{-1}+q^{2Q}\Delta|v|_\infty+q^Q\right)\right),
\end{split}
\end{equation*}
where here and in the sequel, we set
\begin{equation} \label{caL}
\mathcal{L}:=\log q \Delta^{-1}.
\end{equation}
Taking sqare root, the following bound for $P(x)$ emerges. 

\begin{Proposition}\label{FirstPx} We have 
\begin{equation}\label{Pxfirst}
\begin{split}
P\left(\frac{u}{v} +z\right) \ll_q & q^{3Q}\Delta+\left(1+2^{(\mathcal{L}-Q)/2}\right) \times \\ &  
\left(q^{3Q/2}\Delta^{1/2}+q^{Q}|v|_{\infty}^{-1/2}+q^{Q}\Delta^{1/2}|v|_\infty^{1/2}+q^{Q/2}\right).
\end{split}
\end{equation}
\end{Proposition}

\section{Second estimate for $P(x)$}
In this section, we shall prove another estimate for $P(x)$, defined in \eqref{Pdef}, which will follow from a more general estimate for a 
corresponding quantity counting Farey
fractions with denominators from a general set in place of squares. As a by-product, we obtain a large sieve inequality for 
\begin{equation} \label{generalls}
\sum\limits_{f \in \CS} \ \sum\limits_{\substack{r \bmod f\\ (r,f)=1}}\mathrel \bigg 
|\sum\limits_{\substack{g \in \BF_q[t]\\ \deg g\le N}} a_g  e\Big(g\cdot \frac{r}{f}\Big)\bigg |^2,
\end{equation}
where 
$$
\CS\subset \mathcal{B}(0,Q_0)\cap (\BF_q[t]\setminus \{0\})
$$
which we shall assume henceforth. 

In analogy to section \ref{DA}, we let $X_1, \cdots, X_R$ be the sequence of Farey fractions $r/f$  with $f\in \CS$, 
$\deg f \leq Q_0$, $\deg r \leq \deg f-1$ and $(r,f)=1$. Hence, the expression in \eqref{generalls} equals
\begin{align*}
\sum\limits_{r=1}^{R}\mathrel \bigg |\sum\limits_{\substack{g \in \BF_q[t]\\ \deg g\le N}} a_g  e\Big(g\cdot X_r\Big)\bigg |^2.
\end{align*}
Again, we set 
\begin{equation}\label{taudef}
\tau: = \frac{1}{\sqrt \Delta},
\end{equation}
and the $Y_l'$s are chosen to be
\[
\frac{u}{v}+\frac{1}{f_kv^2},
\]
where 
$$
v\in \BF_q[t]\setminus \{0\}, \ |v|_\infty \leq  \tau, \ (u,v)=1, \ \deg u< \deg v.
$$ 
The $f_k$'s are polynomials of degree $k \in \BN$ with $k$ satisfying condition \eqref{kcond}, i.e. 
\begin{equation*} 
K:=\lceil \log_q \tau-\deg v\rceil \leq k\leq \kappa:=\lceil2\log_q \tau-2\deg v\rceil.
\end{equation*}
The above inequality implies that 
\[
\Delta \leq \mathrel \Big| \frac{1}{f_kv^2}\mathrel \Big|_\infty \leq \frac{\sqrt{\Delta}}{|v|_\infty}.
\]
Generalising the notion of $P(x)$ in the previous section for $x\in \BF_q(t)_\infty$, we set 
\begin{equation} \label{Pdef}
P_{\CS}(x): = \sum\limits_{\substack{f\in \CS,\, (r,f)=1\\ |r/f-x|_\infty\leq \Delta}} 1.
\end{equation}
As in section \ref{DA}, it follows that
\[
K'(\Delta) \leq   \max_{1\leq l \leq L } P(Y_l).
\]
So we deduce the following.

\begin{Lemma}\label{DA2}
\begin{align*}
K'(\Delta)& \leq  \max_{\substack{v\in \BF_q[t]\setminus \{0\}\\ \deg v \leq \log_q \tau} }\,\max_{\substack{u\in \BF_q(t)\\ (u,v)=1}}\,  
\max_{\substack{k\\ K\leq k\leq \kappa}}P_{\CS}\Big(\frac{u}{v}+ \frac{1}{f_kv^2}\Big)\\
&\leq  \max_{\substack{v\in \BF_q[t]\setminus \{0\}\\ \deg v \leq \log_q \tau} }\,\max_{\substack{u\in \BF_q(t)\\ (u,v)=1}}\,  
\max_{\substack{z \in \BF_q[t]_\infty \\ \Delta\leq |z|_\infty\leq \sqrt \Delta/|v|_\infty}}P_{\CS}\Big(\frac{u}{v}+ z\Big).
\end{align*}
\end{Lemma}

The next lemma gives an estimate for $P_{\CS}(u/v +z)$ in terms of another quantity $\Pi(y,\delta)$ which will then be transformed further.

\begin{Lemma}\label{Bound4} Suppose that the conditions \eqref{cond1} and \eqref{taudef} are satisfied and also suppose that $|z|_\infty \geq \Delta$. Suppose further that $\delta$ is a natural 
number satisfying
\begin{equation}\label{Eqn3}
\frac{q^{Q_0-1}\Delta}{|z|_\infty}\leq q^\delta < q^{Q_0}.
\end{equation}
Let 
\begin{equation}\label{Eqn4}
J(y, \delta): = \mathcal{B}\Big(zvy, \log_q\big(q^{\delta+1} |vz|_\infty\big)\Big)
\end{equation}
and 
\begin{equation}\label{Eqn5}
\Pi(y, \delta): = \sum\limits_{f \in\CS\cap \mathcal{B}(y,\delta)}  \sum\limits_{\substack{g \in J(y, \delta)\\ g \equiv -uf \bmod v \\ g \neq 0} }1.
\end{equation}
Then 
\[
P_{\CS}\left(\frac{u}{v} +z\right)\ll_q  1 +\frac{1}{q^{\delta}} \int\limits_{\mathcal{B}(0, Q_0)}\Pi(y, \delta)dy.
\]
\end{Lemma}

\begin{proof}  We first isolate the contribution of $f$'s which are associates to $v$, getting
\begin{equation} \label{splitaway}
P_{\CS}(x)\le q-1+P_{\CS}'(x),
\end{equation}
where
$$
P_{\CS}':=\sum\limits_{\substack{f\in \CS,\, (r,f)=1\\ |r/f-x|_\infty\leq \Delta\\ f\not\approx v}} 1.
$$
Here ``$f\not\approx v$'' means that $f$ is not an associate of $v$.
Now we define 
\begin{equation*}
P_{\CS}(x,y,\delta): = \sum\limits_{\substack{f\in \CS \cap \mathcal{B}(y,\delta)\\ (r,f)=1\\ |r/f-x|_\infty\leq \Delta\\ f\not\approx v}}1.
\end{equation*} 
Since $\delta < Q_0$, we have 
\begin{align}
\int\limits_{\mathcal{B}(0, Q_0)}P_{\CS}(x,y,\delta)dy & = \sum\limits_{\substack{f\in \CS \\ (r,f)=1\\ |r/f-x|_\infty\leq \Delta\\ f\not\approx v}}\int\limits_{\mathcal{B}(0, Q_0)\cap \mathcal{B}(f,\delta)} 1dy
\geq q^{\delta} P_{\CS}'(x)
\end{align}
and hence 
\begin{equation}\label{Eqn6}
P_{\CS}'(x)\leq \frac{1}{q^{\delta}}\int\limits_{\mathcal{B}(0,Q_0)}P_{\CS}(x,y,\delta)dy
\end{equation}
whenever $\delta < Q_0$. 

Now, if $|r/f-x|_\infty\leq \Delta$, then 
$$
|r-fx|_\infty\leq \Delta|f|_\infty, \text{ i.e. } \deg(r-fx)\leq \lceil \log_q\Delta\rceil +\deg f.
$$
From this and $x = u/v + z$, we deduce that 
\[
\mathrel\bigg| r -\frac{fu}{v}-fz\mathrel\bigg|_\infty \leq \Delta |f|_\infty
\]
and hence 
\[
|rv - fu - vfz|_\infty \leq \Delta q^{Q_0} |vz|_\infty/ |z|_\infty.
\]
Since $q^{Q_0-1}\Delta/|z|_\infty\leq q^\delta$, it follows that
\begin{equation}\label{Eqn1}
|rv - fu - vfz|_\infty \leq q^{\delta+1} |vz|_\infty.
\end{equation}
If $f\in \mathcal{B}(y,\delta)$, then 
\begin{equation}\label{Eqn2}
|fvz -yvz|_\infty \leq q^{\delta} |vz|_\infty.
\end{equation}
From (\ref{Eqn1}) and (\ref{Eqn2}), we have 
\[
|rv - fu - yvz|_\infty \leq q^{\delta+1} |vz|_\infty.
\]
We observe that $rv - fu \not=0$ because $(r,f) = (u,v) = 1$ and $f\not\approx v$. Writing $g = rv - fu$ and recalling $(\ref{Eqn4})$ and $(\ref{Eqn5})$, we deduce that 
\begin{align*}
P_{\CS}(x,y,\delta) \le \Pi(y, \delta).
\end{align*}
Combining this with \eqref{splitaway} and \eqref{Eqn6}, we obtain the desired result.  
\end{proof}

\textbf{Further notations:} For $ h \in \BF_q[t] \backslash \{ 0\}$ we put
$$
\CS_h : = \{  x\in \BF_q[t]_\infty : hx \in \CS\}.
$$
We note that 
$$
\CS_h \subset \mathcal{B}\big(0, Q_0- \deg h\big). 
$$
We shall require that the number of elements of $\CS_h$ in small sections of arithmetic progressions in $\BF_q[t]$ does not differ too much from the expected number. To measure the distribution of $\CS_h$ in  
sections of arithmetic progressions, we define the quantity 
\begin{equation}\label{Ah}
A_h(m; k, l) : = \max_{\substack{y \in \BF_q(t)_\infty\\ |y|_\infty \leq  Q_0-\deg h }}\mathrel\Big| \mathrel\Big\{ x \in \CS_h \cap \mathcal{B}(y,m)     : x \equiv l \bmod k     \mathrel\Big \}\mathrel\Big|,
\end{equation}
where 
\begin{equation} \label{lotcondis}
0\le m\leq Q_0- \deg h,\,  k \in \BF_q[t]\backslash \{0\},\ |k|_\infty\le \Delta^{-1/2}, \ l \in \BF_q[t], \ (k, l)=1.
\end{equation}

Next we express  $\Pi(y, \delta)$ in terms of  $A_h(m, k, l) $. This will lead us to the following estimate  for $P_{\CS}(u/v +z)$.

\begin{Lemma}\label{DA3} We have 
\begin{equation}\label{Eq8}
\begin{split}
P_{\CS}(u/v +z) \ll_q 1 + 
\sum\limits_{\substack{ h | v}}\sum\limits_{\substack{|g |_\infty\leq q^{Q_0}|h|_{\infty}^{-1}|vz|_\infty\\ ( g, v/h)=1 } }
A_h\bigg(Q_0+  \log_q\Big(\frac{\Delta}{|zh|_\infty}\Big);\frac{v}{h}, -u' g\bigg),
\end{split}
\end{equation}
where $u u' \equiv 1 \bmod v$.
\end{Lemma}

\begin{proof} We split  $\Pi(y, \delta)$ into 
\[
\Pi(y, \delta)= \sum\limits_{h|v} \sum\limits_{\substack{f \in\CS\cap \mathcal{B}(y,\delta)\\ (f,v)=h}}  
\sum\limits_{\substack{g \in J(y, \delta)\\ g \equiv -uf \bmod v \\ g \neq 0} }1,
\]
where $h$ runs over a maximal set of mutually non-associate elements of  $\BF_q[t]\backslash \{0\}$, and $(f, v)= h$ means that $h$ is a greatest common divisor of $(f, v)$ (unique  up to associates). Writing $\tilde f : = f/h$ and $\tilde g : = g/h$, it follows that
\begin{align*}
\Pi(y, \delta)& \leq \sum\limits_{\substack{ h | v}}\sum\limits_{\substack{f/h\in\CS_h\cap \mathcal{B}\left(y/h,\, {\delta}- \deg h\right)\\ 
(f/h, v/h)=1}}  \sum\limits_{\substack{g/h\in J\left(y/h,\delta- \deg h\right)\\ g/h \equiv -uf/h \bmod v/h \\ g/h \neq 0} }1 \\
& =  \sum\limits_{\substack{ h | v}}\sum\limits_{\substack{\tilde f\in\CS_h\cap \mathcal{B}\left(y/h,\delta- \deg h\right)\\ 
(\tilde f, v/h)=1 }}  \sum\limits_{\substack{\tilde g\in J\left(y/h,\delta- \deg h\right)\\ \tilde g \equiv -u\tilde f \bmod v/h \\ 
\tilde g \neq 0} }1 \\
& =  \sum\limits_{\substack{ h | v}} \sum\limits_{\substack{\tilde g\in J\left(y/h,\delta- \deg h\right)\\ 
(\tilde g, v/h)=1 \\ \tilde g \neq 0} }
\sum\limits_{\substack{\tilde f\in\CS_h\cap \mathcal{B}\left(y/h,\delta- \deg h\right)\\ \tilde f \equiv -u'\tilde g \bmod v/h }} 1,
\end{align*}
where $uu'\equiv 1\bmod{v/h}$. 
Hence, by definition of $A_h(m; k, l)$ in (\ref{Ah}), we have 
\[
\Pi(y, \delta) \leq \sum\limits_{\substack{ h | v}} \sum\limits_{\substack{ g\in J\left(y/h,\delta- \deg h\right)\\ 
( g, v/h)=1 \\ \ g \neq 0} }A_h\left({\delta}- \deg h;\frac{v}{h}, -u' g\right).
\]
Integrating the last line over $y$ in the ball $\mathcal{B}(0, Q_0)$ and rearranging the order of summation and integration, we get
\begin{align*}
& \int\limits_{\mathcal{B}(0, Q_0)}\Pi(y, \delta)dy \\ \leq & \displaystyle \sum\limits_{\substack{ h | v}}\int\limits_{\mathcal{B}(0, Q_0)} 
\sum\limits_{\substack{ g\in J\left(y/h,\delta- 
\deg h\right)\\ ( g, v/h)=1 \\ \ g \neq 0} }A_h\Big({\delta}- \deg h;\frac{v}{h}, -u' g\Big)dy \\
\leq & \displaystyle \sum\limits_{\substack{ h | v}}\int\limits_{\mathcal{B}(0, Q_0)}\sum\limits_{\substack{\left|g -vyz/h\right|_\infty\leq 
q^{\delta+1}|vz|_\infty/|h|_\infty
\\ ( g, v/h)=1 \\ \ g \neq 0} }A_h\Big({\delta}- \deg h;\frac{v}{h}, -u' g\Big)dy \\
\leq & \sum\limits_{\substack{ h | v}}\sum\limits_{\substack{0<|g |_\infty\leq q^{Q_0}|vz|_\infty/|h|_\infty
\\ ( g, v/h)=1 } }A_h\Big({\delta}- \deg h;\frac{v}{h}, -u' g\Big) \times\\ &
\int\limits_{\mathcal{B}(0, Q_0)\cap \mathcal{B}\left(gh/(vz),\delta+1\right)} 1dy\\
\le & q^{\delta+2}\displaystyle \sum\limits_{\substack{ h | v}}\sum\limits_{\substack{0<|g |_\infty\leq q^{Q_0}|vz|_\infty/|h|_\infty
\\ ( g, v/h)=1} }A_h\Big({\delta}- \deg h;\frac{v}{h}, -u' g\Big). 
\end{align*}
Now choosing $\delta$ such that
\[
\frac{q^{Q_0-1}\Delta}{|z|_\infty}\le q^{\delta}< \frac{q^{Q_0}\Delta}{|z|_\infty}\le q^{\delta}
\]
 and using Lemma \ref{Bound4}, we obtain (\ref{Eq8}).
\end{proof}

If we assume the set $\CS_h$ to be evenly distributed in the residue classes $l \bmod k$, then if $\mathcal{B}(y, m) \subset \mathcal{B}\big(0, Q_0-\log_q |h|_\infty\big)$, the expected cardinality of the set 
\[
\mathrel\Big\{ x \in \CS_h \cap \mathcal{B}(y,m)     : x \equiv l \bmod k     \mathrel\Big \}
\]
is 
\[
\asymp \frac{|\CS_h|/|k|_\infty}{q^{Q_0}/|h|_\infty}\cdot q^m.
\]
This suggests to set a condition of the form 
\begin{equation}\label{Eq9}
A_h(m, k, l)\leq \mathrel\Big(1+ \frac{|\CS_h|/|k|_\infty}{q^{Q_0}/|h|_\infty}\cdot q^m\mathrel\Big)X,
\end{equation}
where $X\geq 1$ is thought to be small compared to $q^{Q_0}$ and $q^N$. Under the condition \eqref{Eq9},
we shall infer the following bound from Lemma \ref{DA3}. 

\begin{Lemma}\label{Bound5} Suppose the condition \eqref{Eq9} to  hold for all $h,k,l,m$ satisfying \eqref{lotcondis}.
Then 
\begin{equation}
P_{\CS}\left(\frac{u}{v}+z\right)\ll_q 1+q^{Q_0}X2^{\deg v}(|vz|_\infty+|\CS|\Delta).
\end{equation}
\end{Lemma}

\begin{proof} 
Equations \eqref{Eq8} and \eqref{Eq9} imply 
\begin{align*}
 \sum\limits_{\substack{ h | v}}& \sum\limits_{\substack{0<|g |_\infty\leq {q^{Q_0}}|vz|_\infty/{|h|_\infty}\\ ( g,v/h)=1 } }A_h\bigg(Q_0+  \log_q\Big(\frac{\Delta}{|zh|_\infty}\Big),\frac{v}{h}, -u' g\bigg)\\
&\leq\sum\limits_{\substack{ h | v}} \sum\limits_{\substack{0<|g |_\infty\leq {q^{Q_0}}|vz|_\infty/{|h|_\infty}\\ ( g, v/h)=1 } }\mathrel\bigg( 1 + \frac{|\CS_h|/|v/h|_\infty}{q^{Q_0}/|h|_\infty}
\cdot \frac{ q^{Q_0}\Delta}{|zh|_\infty}\mathrel\bigg)X\\
& \ll \sum\limits_{\substack{ h | v}}\mathrel\bigg( 1 + \frac{|\CS_h||{h}|_\infty\Delta}{|vz|_\infty}\mathrel\bigg)\cdot 
\frac{q^{Q_0}|vz|_\infty}{|h|_\infty}\cdot X \\
&=  \sum\limits_{\substack{ h | v}}\mathrel\bigg( \frac{|vz|_\infty}{|h|_\infty}+|\CS_h|\Delta\mathrel\bigg)q^{Q_0} X\\
&\le q^{Q_0}X\tau(v)(|vz|_\infty+|\CS|\Delta).
\end{align*}
This together with \eqref{tauest} and Lemma \ref{DA3} gives the desired result. 
\end{proof}

Upon choosing $\Delta:=q^{-N}$, Lemmas \ref{LS2}, \ref{DA2} and \ref{Bound5} imply the following general large sieve inequality for function 
fields which is an analogue of 
\cite[Theorem 2]{Bai} for the classical case. 

\begin{Theorem}\label{LSsparse} Suppose the condition \eqref{Eq9} to  hold for all $h,k,l,m$ satisfying \eqref{lotcondis}.
Then 
\begin{equation}
\begin{split}
& \sum\limits_{f \in \CS} \ \sum\limits_{\substack{r \bmod f\\ (r,f)=1}}\mathrel \bigg 
|\sum\limits_{\substack{g \in \BF_q[t]\\ \deg g\le N}} a_g  e\Big(g\cdot \frac{r}{f}\Big)\bigg |^2 \\ \ll_q & 
\Big(q^{N} + Q_0X2^{N/2}\big(q^{N/2}+|\CS|)\Big)\sum\limits_{\substack{g \in \BF_q[t]\\ \deg g\le N}}| a_g|^2.
\end{split}
\end{equation}
\end{Theorem}

However, this corrected version of \cite[Corollary 5.2]{BaSi} (Claim 1 in section 3 of the present paper), is only a by-product in this paper. 
Next, we specialize $\CS$ to square moduli and derive the following estimate for $P(x)$, as defined in \eqref{Pdef}, from Lemma \ref{Bound5}.

\begin{Proposition} \label{secondPx} We have 
\begin{equation} \label{claimed}
P\left(\frac{u}{v}+z\right) \ll_q  1+2^{\mathcal{L}}q^{2Q}\left(\Delta^{1/2}+q^{Q}\Delta\right),
\end{equation}
where $\mathcal{L}$ is defined as in \eqref{caL}.
\end{Proposition}

\begin{proof} We shall apply Lemma \eqref{Bound5} with $Q_0=2Q$ and $\mathcal{S}$ the set of all squares $s$ of norm $q^{2Q}$. All we need
to do is to work out the size of $X$ in condition \eqref{Eq9}. This is completely parallel to the classical case, which has been worked out in 
\cite{Bai}. First, let
$$
h = \epsilon P_1^{v_1}\cdots P_n^{v_n}
$$
be the unique prime factorization of $h$ with $P_1,...,P_n\in \mathbb{F}_q[t]$ monic irreducible polynomials and 
$\epsilon\in \mathbb{F}_q^{\ast}$. 
For $i = 1, ..., n$ let
$$
u_i:= \begin{cases} v_i & \mbox{ if } v_i \mbox{ is even},\\
v_i + 1 & \mbox{ if } v_i \mbox{ is odd.}
\end{cases}
$$
Put
$$
F_h:= P_1^{u_1/2}\cdots P_n^{u_n/2}
$$
Then $R = R_1^2\in \mathcal{S}$ is divisible by $h$ if and only if $R_1$ is divisible by $F_h$. 
Thus
$$
\mathcal{S}_h = \left\{ R_2^2G_h \ :\ \deg R_2 = Q-\deg F_h \right\} \subset \left\{a \ :\ \deg a=2Q-\deg h\right\}
$$
where
$$
G_h:=\frac{F_h^2}{h}=\epsilon^{-1} P_1^{u_1-v_1}\cdots P_n^{u_1-v_1}.
$$
Hence,
$$
|S_h| \le q^{Q-\deg F_h+1}.
$$

Let $\delta_h(k, l)$ be the number of solutions $x$ mod $k$ to the congruence
\begin{equation} \label{co}
x^2G_h \equiv l \bmod k.
\end{equation}
Then it follows that condition \eqref{Eq9} holds true
for all positive $m \le  2Q-\deg h$ and
$$
X=\delta_h(k,l).
$$ 
Thus the remaining task is to bound $\delta_h(k,l)$.

If $(G_h,k)>1$, then $\delta_h(k, l) = 0$ since $k$ and $l$ are supposed to be coprime.
Therefore, we can assume that $(G_h,k) = 1$. Let $G$ mod $k$ be a multiplicative
inverse of $G_h \bmod{k}$, i.e. $GG_h \equiv 1 \bmod k$. Put $l^{\ast}:= Gl$. Then \eqref{co} is equivalent
to
\begin{equation*}
x^2 \equiv l^{\ast} \bmod{k}.
\end{equation*}
Taking into account that $(k,l^{\ast})= 1$, this congruence has at most two solutions if $k$ is a power
of an irreducible polynomial, where we recall that $q$ is not a power of $2$. From this it
follows using the Chinese remainder theorem that for all $k\in \mathbb{F}_q[t]$ we have 
$$
\delta_h(k, l) \le 2^{\omega(k)},
$$
where $\omega(k)$ is the number of distinct monic irreducible factors of $k$. For $|k|_\infty \le \Delta^{-1/2}$, we have
$$
\omega(k)\le \deg k\le \log_q\Delta^{-1/2}= \frac{\mathcal{L}}{2}.
$$
Therefore, \eqref{Ah} holds with 
$$
X:=2^{\mathcal{L}/2}.
$$
Now the claimed inequality follows from Lemma \ref{Bound5} upon recalling that
$$
|vz|_\infty \le \Delta^{1/2}
$$
and
$$
\deg v = \log_q |v|_\infty \le \log_q\Delta^{-1/2}=\frac{\mathcal{L}}{2}.
$$
\end{proof}

\section{Further transformation of $P(x)$}
In this section we transform $P(x)$ further by an application of Poisson summation. We then derive a third estimate for $P(x)$ 
which, in certain ranges, is better than the previously proved ones.

Throughout the following, we  suppose that $|z|_\infty \geq \Delta$. We further assume that $Q_0$ is even and set
$$
Q:=\frac{Q_0}{2}.
$$
Then applying Lemma \ref{Bound4} with  $x$ of the form  in \eqref{cond2},
$\delta$ a real  parameter satisfying \eqref{Eqn3} and 
$$
\CS : = \{ h^2 : h\in \BF_q[t],\ \deg h = Q\},
$$
we have 
\begin{equation} \label{Pbound}
P_{\CS}\left(\frac{u}{v} +z\right)\ll_q  1 +\frac{1}{q^{\delta}} \int\limits_{\mathcal{B}(0, 2Q)}\Pi(y, \delta)dy,
\end{equation}
where 
\begin{equation}
\Pi(y, \delta): = \sum\limits_{\substack{\deg h = Q\\ 
|h^2 -y|_\infty\leq  q^\delta}}  \sum\limits_{\substack{g \in J(y, \delta)\\ g \equiv -uh^2 \bmod v \\ g \neq 0} }1.
\end{equation}

Recall the notations in section 7. If $y\not\in \BF_q(t)_{\infty}^2$, then 
$$
\deg(x^2-y)=\max\{2\deg x,\deg y\}
$$
for every $x\in \BF_q(t)_{\infty}$. Hence, in this case,
$$
\deg(h^2-y)=2Q>\delta
$$
and therefore
$$
\Pi(y,\delta)=0.
$$

If $y\in \BF_q(t)_{\infty}^2$,
then the conditions $|h^2 -y|_\infty\leq  q^\delta$ and $\deg h=Q$ imply that $\deg y=2Q$ and 
$$
|h -\sqrt y|_\infty\leq  q^{\delta-Q} \mbox{ or }  |h +\sqrt y|_\infty\leq  q^{\delta-Q}.
$$
It follows that
\begin{equation*}
\begin{split}
\Pi(y, \delta) \ll \Pi_1(y,\delta)+\Pi_2(y,\delta),
\end{split}
\end{equation*}
where, for $i=1,2$, $\Pi_i(y,\delta)=0$ if $y\not\in \BF_q(t)_{\infty}^2$ and 
\begin{equation*}
\Pi_i(y, \delta) := \frac{q^{Q}}{2|\sqrt{y}|_{\infty}} \sum\limits_{|h +(-1)^i \sqrt y|_\infty\leq  q^{\delta-Q}} \ 
\sum\limits_{\substack{|g -yvz|_\infty\leq q^{\delta+1}|vz|_\infty\\ g \equiv -uh^2 \bmod v } }1 
\end{equation*}
if $y\in \BF_q(t)_{\infty}^2$. 

It follows that 
\begin{align}\label{Dsum1}
\Pi_i(y, \delta) \leq \frac{q^{Q}}{2|\sqrt{y}|_{\infty}} \cdot \Sigma_i,
\end{align}
where 
\begin{equation} \label{Sigma_i}
\Sigma_i : =   
\sum\limits_{h\in \BF_q[t]} \Phi_{1}\left( \frac{h+(-1)^i \sqrt y}{t^{\delta-Q+1}}\right)\sum\limits_{\substack{g \in \BF_q[t]\\ g \equiv -uh^2 \bmod v } }
\Phi_{1}\left( \frac{g-yvz}{vz t^{\delta+1}}\right)
\end{equation}
for $i=1,2$. Here $\Phi_1(x)$ is defined as in \eqref{Phi1def}.  

Applying the Poisson summation formula, Lemma \ref{Poisson},  with a linear change of variable to the sum over $g$, and 
using $\Phi_1=\hat{\Phi}_1$ (see Lemma \ref{Phi1}), we transform the inner-most sum in \eqref{Sigma_i} into  
\[
\sum\limits_{\substack{g \in \BF_q[t]\\ g \equiv -uh^2 \bmod v } }\Phi_{1}\left( \frac{g-yvz}{vz t^{\delta+1}}\right) = |z|_\infty q^{\delta+1} \sum\limits_{\ell\in \BF_q[t]}\Phi_1(zt^{\delta+1} \ell)e\left(yz\ell +\frac{uh^2\ell}{v}\right).
\]
It follows that 
\begin{equation}\label{Dsum2}
\begin{split}
\Sigma_i= & |z|_\infty q^{\delta+1}\sum\limits_{\ell\in \BF_q[t]}\Phi_1(zt^{\delta+1} \ell)e\left(yz\ell \right)\sum\limits_{s \bmod v^\ast} 
e\left(\frac{us^2\ell^\ast}{v^\ast}\right) \times\\ & \sum\limits_{\substack{r\in \BF_q[t]\\ r =s \bmod v^\ast}}
\Phi_{1}\left( \frac{r+(-1)^i\sqrt y}{t^{\delta-Q+1}}\right)
\end{split}
\end{equation}
for $i=1,2$,
where 
\begin{equation} \label{ast}
v^\ast : = v/(v,\ell) \quad \mbox{and} \quad  \ell^\ast : = \ell/(v,\ell).
\end{equation}
Again applying the Poisson summation formula, Lemma \ref{Poisson}, 
with a linear change of variable to the sum over $r$, and using $\Phi_1=\hat{\Phi}_1$, we transform the inner-most sum into  
\begin{equation}\label{Insum2}
\begin{split}
& \sum\limits_{\substack{r\in \BF_q[t]\\ r =s \bmod v^\ast}}\Phi_{1}\left( \frac{r+(-1)^i\sqrt y}{t^{\delta-Q+1}}\right)\\ = &
\frac{q^{\delta-Q +1}}{|v^\ast|_\infty} \sum\limits_{b\in \BF_q[t]}
\Phi_1\left(\frac{t^{\delta-Q+1}b}{v^\ast}\right)e\left(b\cdot \frac{-(s+(-1)^i\sqrt y)}{v^\ast}\right).
\end{split}
\end{equation}
Combining \eqref{Pbound}, \eqref{Dsum1}, \eqref{Dsum2} and \eqref{Insum2}, we obtain
\begin{equation} \label{Dsum5}
\begin{split}
& P_{\CS}\left(\frac{u}{v} +z\right) \ll_q 1+
|z|_\infty\cdot q^{\delta} \cdot \Bigg| \sum\limits_{\ell\in \BF_q[t]}
\frac{\Phi_1(zt^{\delta+1} \ell)}{|v^\ast|_\infty} \\ & 
\sum\limits_{b\in \BF_q[t]}\Phi_1\left(\frac{t^{\delta-Q+1}b}{v^\ast}\right)\cdot
G\big(u\ell^\ast,-b,v^\ast\big) \cdot \left(E\Big(z\ell,\frac{b}{v^\ast}\Big)+E\Big(z\ell,-\frac{b}{v^\ast}\Big)\right)\Bigg|
\end{split}
\end{equation}
for $i=1,2$, 
where the quadratic Gauss sum $G(\alpha,l;\beta)$ and the exponential integral $E(A,B)$ are defined as in \eqref{Gsum} and \eqref{expindef}, 
respectively. 

\section{Treatment of simple cases}
In this section, we estimate the contributions to \eqref{Dsum5} which can be treated easily.

Note that $v^\ast =1$ if $\ell =0$. It follows that the contribution of $\ell=0=b$ is bounded 
by
\begin{equation*} 
\ll_q q^{\delta+Q}|z|_\infty,
\end{equation*}
and the contribution of $\ell = 0$, $b \neq  0$ vanishes using Lemma \ref{expinte1}. 

Now we consider the case when $\ell \neq  0$ and $|b/(v^{\ast}z\ell)|_\infty > q^{Q}$. To apply the results on quadratic exponential
integrals in section \ref{QEIntegrals}, we set 
$$
A:=z\ell \quad \mbox{and} \quad B:=\pm \frac{b}{v^{\ast}}.
$$
Then from Lemma \ref{expinte2}, we deduce that
\begin{equation} \label{EABesti}
E(A,B)\ll |z\ell|_{\infty}^{-1/2} \cdot \left| \int\limits_{B(x,n)} e(\alpha y^2)dy \right|,
\end{equation}
where
$$
x=\frac{b}{2v^{\ast}\sqrt{z\ell\alpha}} 
$$
and 
\begin{equation} \label{udef}
n= Q+\deg \sqrt{\frac{z\ell}{\alpha}}. 
\end{equation}
We note that we are here in the case when $\deg x >n$. Now taking Lemma \ref{expinte3} into consideration, 
the integral on the right-hand side of \eqref{EABesti} is zero unless $n<-\deg x$, which is equivalent to 
$$
\deg b < -Q+\deg v^{\ast}+\deg \alpha.
$$
Hence, we possibly have a non-zero contribution only if 
$$
|b|_{\infty}\le q^{1-Q}|v^{\ast}|_{\infty},  
$$
in which case we use the trivial estimate 
$$
E(A,B) \ll q^{Q}
$$
(which is essentially the same as what we get when combining \eqref{EABesti} and Lemma \eqref{expinte3}).  
Using Lemma \eqref{Gsum5}, it follows that the total contribution in this case is bounded by 
\begin{equation*}
\ll_q \  q^{\delta}|z|_{\infty}\cdot \sum\limits_{\substack{\ell\not=0\\ |\ell|_\infty\le q^{-\delta}|z|_{\infty}^{-1}}} 
|v^{\ast}|_{\infty}^{-1/2}
\ \sum\limits_{|b|_\infty\le q^{1-Q}|v^{\ast}|_{\infty}} q^{Q}  
\ll_q  |v|_{\infty}^{1/2}.
\end{equation*}

Finally, we consider the case when $\ell \neq  0$ and $|b/(v^{\ast}z\ell)|_\infty \le q^{Q}$ and $n \leq -1$, where $n$ is defined as in
\eqref{udef}.
In this case we use the trivial estimate 
$$
E(A,B) \ll_q |z\ell|^{-1/2}_\infty.
$$
Also note that the condition $n \leq -1$ implies
\[
|\ell|_\infty \leq q^{-2Q-3}|z|^{-1}_{\infty}.
\]
Hence, the contribution of this case to \eqref{Dsum5} is bounded by 
\begin{equation}\begin{split}
\ll_q q^{\delta}|z|_\infty\cdot  \sum\limits_{\substack{\ell\not=0 \\ 
|\ell|_\infty\leq q^{-2Q-3}|z|^{-1}_{\infty}}} |v^\ast|^{-1/2}_\infty \sum\limits_{|b|_\infty\leq |v^\ast|_\infty q^{-\delta+Q-2}} 
|z\ell|^{-1/2}_\infty \ll_q |v|^{1/2}_\infty.
\end{split}
\end{equation}
Consequently, the total contribution to  \eqref{Dsum5} of the above three cases is
\begin{equation} \label{threecases}
\ll_q  q^{\delta+Q}|z|_\infty + |v|_{\infty}^{1/2}.
\end{equation}

\section{Treatment of critical case}
It remains to consider the critical case when $\ell \neq  0$ and $|b/(v^{\ast}z\ell)|_\infty \le q^{Q}$ and $n\geq 0$ 
in which we perform a precise evaluation of the Gauss sums and exponential integrals and then transform the resulting exponential sums further. 

As in the last section we set 
$$
A:=z\ell \quad \mbox{and} \quad B:=\pm \frac{b}{v^{\ast}}.
$$
Then from Lemma \ref{expinte2}, we deduce that
\begin{equation}
E(A,B)= q^{-\lceil(\deg z\ell)/2\rceil}\cdot e\left(\frac{-b^2}{4{v^\ast}^2 z\ell}\right)\cdot \int\limits_{\mathcal{B}(0,n )} e\left( \alpha y^2 \right) dy,
\end{equation}
where $n$ is defined as in \eqref{udef}.
In the case $n \geq 0$, Lemma \ref{expinte3} gives 
\begin{align*}
 \int\limits_{\mathcal{B}(0,n )} e\left( \alpha y^2 \right) dy = \begin{cases} 1  & \mbox{ if} \quad \epsilon =0 \\ q^{1/2} & 
 \mbox{ if} \quad \epsilon =1 \quad \mbox{and $c$ is a square} \\  -q^{1/2} & \mbox{ if} \quad \epsilon =1 \quad \mbox{and $c$ is 
 not a square}. 
\end{cases}
\end{align*}
Now we define
\begin{equation}
\sigma(\ell) := \begin{cases}1 &\mbox{ if  $z\ell$ has even degree or} \\ & \mbox{ $z\ell$ has odd degree and $c$ is a  square} \\
-1  &\mbox{ if $z\ell$ has odd degree and $c$ is not a  square}.
\end{cases}
\end{equation}
Then it follows that
\begin{equation*}
E(A,B) = e\left(\frac{-b^2}{4{v^\ast}^2 z\ell}\right)\cdot {|z\ell|_\infty^{-1/2}}\cdot \sigma(\ell).
\end{equation*}
\noindent
We note that the condition $n \geq 0$ is equivalent to
$$
\deg z\ell \geq -2-2Q.
$$
Hence, the contribution of this case to the right-hand side of \eqref{Dsum5} is bounded by
\begin{equation}\begin{split}
\ll & q^{\delta}|z|_{\infty}\cdot 
\sum\limits_{\substack{\ell\not=0\\ |\ell|_\infty \leq q^{-\delta}|z|_\infty^{-1}}}  |v^\ast|_\infty^{-1}\cdot  \Bigg|
\sum\limits_{|b|_\infty\leq q^{Q}|v^\ast z\ell|_\infty} G\big(u\ell^\ast,-b,v^\ast\big) \times\\ 
& e\left(\frac{-b^2}{4{v^\ast}^2 z\ell}\right)\cdot |z\ell|_\infty^{-1/2}\cdot\sigma(\ell) \Bigg|.
\end{split}
\end{equation} 
Using Lemmas \ref{Gsum1} and \ref{Gsum5}, we bound the above double sum by 
\begin{equation} \label{hardcase} \begin{split}
\ll q^{\delta}|z|_{\infty}^{1/2}  \cdot
\sum\limits_{\substack{\ell\not=0\\ |\ell|_\infty \leq q^{-\delta}|z|_\infty^{-1}}} |v^\ast \ell|_{\infty}^{-1/2}\cdot 
\Bigg|
\sum\limits_{\deg b \leq M} e\left(Vb^2\right)\Bigg|,
\end{split}
\end{equation}
where we set
\begin{equation}
V:=\frac{\ol{u\ell^\ast}}{4v^\ast}
+\frac{1}{4{v^\ast}^2 z\ell}
\end{equation}
and 
\begin{equation} \label{Mdef}
M:=Q+\deg v^{\ast}+\deg z + \deg \ell,
\end{equation}
keeping in mind that $V$ and $M$ depend on $\ell$. We note that 
\begin{equation} \label{Mbound}
M\le Q+\mathcal{L}, 
\end{equation}
where $\mathcal{L}$ is defined as in \eqref{caL}.

\section{Simplification of the quadratic exponential sum}
In the following, we simplify the exponential sum over $b$ in \eqref{hardcase}. First we rewrite $V$ in a more suitable form. Set
\begin{equation} \label{fkast}
f_k^\ast = \frac{1}{zv^\ast v}.
\end{equation}
From \eqref{cond2}  and the fact that $v^\ast | v$, we have $f_k^\ast \in \BF_q[t]\setminus \{0\}$. We further assume that 
\[
\ol u \equiv -a  \bmod v^\ast, \quad \deg a < \deg v^\ast.
\]
Using the reciprocity relation
\[
\frac{\ol{\ell^\ast}}{v^\ast} \equiv - \frac{\ol{v^\ast}}{\ell^\ast} +\frac{1}{\ell^\ast v^\ast} \bmod{1}
\]
for Kloosterman fractions and the relation $v^\ast \ell = \ell^\ast v$,
we deduce that
$$
V\equiv \frac{a\overline{v^\ast} +f_k^\ast}{4\ell^\ast}-\frac{a}{4v^\ast \ell^\ast} \bmod{1}.
$$

Next, we remove the term $a/(4v^\ast \ell^\ast)$ using summation by parts. 
We arrange the $b$'s in question into a sequence $b_1,b_2,...,b_N$ satisfying 
$$
b_1=0 \quad \mbox{ and } \quad  
|b_{i+1}-b_{i}|_{\infty}=q^{\ord_q(i)} \mbox{ for } i=1,...,N-1,
$$
where $N:=q^{M+1}$ and 
$$
\ord_q(i)=\max\limits_{q^{\alpha}|i} \alpha. 
$$ 
Now we write
\begin{equation*}
\begin{split}
& \sum\limits_{\deg b \leq M} e\left(Vb^2\right) = 
\sum\limits_{i=1}^{N} e\left(-\frac{a}{4v^\ast \ell^\ast}\cdot b_i^2\right) \cdot e\left(\frac{a\overline{v^\ast} +f_k^\ast}{4\ell^\ast}
\cdot b_i^2\right)\\
= & e\left(-\frac{a}{4v^\ast \ell^\ast}\cdot b_{N}^2\right) \cdot 
\sum\limits_{j=1}^{N} e\left(\frac{a\overline{v^\ast} +f_k^\ast}{4\ell^\ast}\right)- \\
& \sum\limits_{i=1}^{N-1} \left(e\left(-\frac{a}{4v^\ast \ell^\ast}\cdot b_{i+1}^2\right)-
e\left(-\frac{a}{4v^\ast \ell^\ast}\cdot b_{i}^2\right)\right) 
\cdot \sum\limits_{j=1}^{i} e\left(\frac{a\overline{v^\ast} +f_k^\ast}{4\ell^\ast}
\cdot b_j^2\right).
\end{split}
\end{equation*}
We bound the differences of exponentials above by 
\begin{equation*}
\begin{split}
& \left|e\left(-\frac{a}{4v^\ast \ell^\ast}\cdot b_{i+1}^2\right)-
e\left(-\frac{a}{4v^\ast \ell^\ast}\cdot b_{i}^2\right)\right| = 
\left|e\left(-\frac{a}{4v^\ast \ell^\ast}\cdot (b_{i+1}^2-b_i^2)\right)-1 \right|\\
\ll & \left|\frac{a}{4v^\ast \ell^\ast}\cdot (b_{i+1}^2-b_i^2)\right|_{\infty}\le q^M |\ell^{\ast}|^{-1}_{\infty} |b_{i+1}-b_i|_{\infty}. 
\end{split}
\end{equation*}
We calculate that
$$
\sum\limits_{i=1}^{N-1} |b_{i+1}-b_i|_{\infty} = Mq^{M+1}. 
$$
Hence, we deduce that
\begin{equation} \label{butzelbatz}
\sum\limits_{\deg b \leq M} e\left(Vb^2\right) \ll \left(1+Mq^{2M} |\ell^{\ast}|^{-1}\right)
\sup_{1\le i\le N} \left|\sum\limits_{j=1}^{i} e\left(Wb_j^2\right)\right|,
\end{equation}
where we set
$$
W:=\frac{a\overline{v^\ast} +f_k^\ast}{4\ell^\ast}
$$
for convenience. 

\section{Application of Weyl shift}
It remains to estimate the partial sums 
$$
\sum\limits_{j=1}^{i} e\left(Wb_j^2\right).
$$
Taking the ordering of the $b_j$'s into account, we split this sum into $O(Mq)$ subsums of the form
$$
\sum\limits_{\substack{|b-B|_{\infty}\le q^m}} e\left(Wb\right),
$$
where $-1\le m\le M-1$ and $q^{m+1}\le |B|_\infty\le q^M$.
 
To bound the above subsums, we apply a Weyl shift to the modulus square, getting
\begin{align*}
\Bigg|\sum_{\substack{|b-B|_{\infty}\le q^m}} e\left(Wb^2\right) \Bigg|^2 = & 
\sum_{\substack{b,\tilde{b}\\ |b-B|_{\infty}\le q^m\\ |\tilde{b}-B|_{\infty}\le q^m}} e\left(W(\tilde b-b)(\tilde b+b)\right)\\
= & 
\sum\limits_{|h|_{\infty} \leq q^m}  \ \sum_{|c-B|_{\infty}\le q^m} e\left(Whc\right),
\end{align*}
where we write $c=\tilde{b}+b$ and $h=\tilde{b}-b$. 
Defining $\Phi_1$ as in \eqref{Phi1def} and using the Poisson summation formula (Lemma \ref{Poisson}) together with $\Phi_1=\hat{\Phi}_1$ (Lemma
\ref{Phi1}), we have 
\begin{align*}
 \sum\limits_{|c-B|_{\infty}\le q^m}  e\left(Whc\right) = & e(WhB) \sum_{\tilde{c}\in\BF_q[t]}e\left(Wh\tilde{c}\right)\Phi_1(t^{-m-1}\tilde{c})
 \\
= & q^{m+1}e(WhB)\sum_{x\in\BF_q[t]}  \Phi_1\big(t^{m+1}(x-Wh)\big)\\
\ll &  \begin{cases}
q^{m+1} &  \text{ if}\, \, \|Wh\| \leq q^{-m-2} \\
0 & \text{ otherwise}.
\end{cases}
\end{align*}

Taking square root, we deduce that
\begin{equation} \label{squaresum}
\Bigg|\sum_{\substack{|b-B|_{\infty}\le q^m}} e\left(Wb^2\right) \Bigg|
\ll \Bigg(q^{m+1}\sum_{\substack{|h|_{\infty}\le q^m\\ \|Wh\|\le q^{-m-2}}} 1\Bigg)^{1/2}.
\end{equation}
We observe that the sum on the right-hand is bounded by
\begin{equation} \label{squaresum2}
q^{m+1} \sum_{\substack{|h|_{\infty}\le q^m\\ \|Wh\|\le q^{-m-2}}} 1 \le \sum_{|h|_{\infty}\le q^M} \min\{q^M,\|Wh\|^{-1}\}.
\end{equation}

Using \eqref{Mbound}, \eqref{butzelbatz}, \eqref{squaresum} and \eqref{squaresum2}, we now bound
\eqref{hardcase} by
\begin{equation} \label{hardcase2} 
\begin{split}
\ll_q & (Q+\mathcal{L})^2 q^{\delta}|z|_{\infty}^{1/2}  \cdot
\sum\limits_{\substack{\ell\not=0\\ |\ell|_\infty \leq q^{-\delta}|z|_\infty^{-1}}} |v^\ast \ell|_{\infty}^{-1/2}  
\left(1+q^{2M} |\ell^{\ast}|^{-1}\right)\times \\
& \Bigg(\sum_{|h|_{\infty}\le q^M} \min\{q^M,\|Wh\|^{-1}\}\Bigg)^{1/2}.
\end{split}
\end{equation}
Recalling the definitions of $v^{\ast}$ and $\ell^{\ast}$ in \eqref{ast}, the above is dominated by
\begin{equation} \label{hardcase1}
\begin{split}
\ll_q & (Q+\mathcal{L})^2 q^{\delta}|z|_{\infty}^{1/2}  \cdot \sum\limits_{d|v} \Sigma_d,
\end{split}
\end{equation}
where 
$$
\Sigma_d:=\sum\limits_{\substack{(\ell^{\ast},v^{\ast})=1\\
0<|\ell^{\ast}|_\infty \leq q^{-\delta}|zd|_\infty^{-1}}} 
|v\ell^{\ast}|_{\infty}^{-1/2}  
\left(1+q^{2M} |\ell^{\ast}|^{-1}\right)\Bigg(\sum_{|h|_{\infty}\le q^M} \min\{q^M,\|Wh\|^{-1}\}\Bigg)^{1/2}
$$
and
\begin{equation} \label{Mrewrite}
M=Q+\deg v+\deg z +\deg \ell^{\ast}.
\end{equation}
Applying the Cauchy-Schwarz inequality, we bound $\Sigma_d^2$ by
\begin{equation} \label{CSch} 
\begin{split}
\Sigma_d^2 \le & \sum\limits_{\substack{(\ell^{\ast},v^{\ast})=1\\
0<|\ell^{\ast}|_\infty \leq q^{-\delta}|zd|_\infty^{-1}}} 
|v\ell^{\ast}|_{\infty}^{-1}  
\left(1+q^{4M} |\ell^{\ast}|^{-2}\right)\times \\
& \sum\limits_{\substack{(\ell^{\ast},v^{\ast})=1\\
0<|\ell^{\ast}|_\infty \leq q^{-\delta}|zd|_\infty^{-1}}}
\sum_{|h|_{\infty}\le q^M} \min\{q^M,\|Wh\|^{-1}\}\\
\ll_q & \left((Q+\mathcal{L})|v|_\infty^{-1}+q^{4Q-2\delta}|v|_{\infty}^3|z|_{\infty}^2|d|_{\infty}^{-2}\right)\times\\
& \sum\limits_{\substack{(\ell^{\ast},v^{\ast})=1\\
0<|\ell^{\ast}|_\infty \leq q^{-\delta}|zd|_\infty^{-1}}}
\sum_{|h|_{\infty}\le q^M} \min\{q^M,\|Wh\|^{-1}\}.
\end{split}
\end{equation}

\section{Final Count}
We bound the double sum over $\ell^{\ast}$ and $h$ on the right-hand side of \eqref{CSch} in the form
\begin{equation} \label{div}
\begin{split}
& \sum\limits_{\substack{(\ell^{\ast},v^{\ast})=1\\
0<|\ell^{\ast}|_\infty \leq q^{-\delta}|zd|_\infty^{-1}}}
\sum_{|h|_{\infty}\le q^M} \min\{q^M,\|Wh\|^{-1}\}\\ \le & \sum\limits_{\substack{(\ell^{\ast},v^{\ast})=1\\
0<|\ell^{\ast}|_\infty \leq q^{-\delta}|zd|_\infty^{-1}}} \ \sum\limits_{\substack{|h|_{\infty}\le q^M\\ 
(a\overline{v^\ast} +f_k^\ast)h \equiv 0 \bmod{\ell^{\ast}}}} q^M + \\ & 
\sum\limits_{0\le j\le \mathcal{L}} \sum\limits_{\deg \alpha \le j}  
\sum\limits_{\substack{(\ell^{\ast},v^{\ast})=1\\
0<|\ell^{\ast}|_\infty \leq q^{-\delta}|zd|_\infty^{-1}}}\ \sum\limits_{\substack{|h|_{\infty}\le q^M\\ 
(a\overline{v^\ast} +f_k^\ast)h \equiv \alpha \bmod{\ell^{\ast}}}} |\ell^{\ast}|_{\infty}q^{-j}\\
\ll_q & \sum\limits_{\substack{(\ell^{\ast},v^{\ast})=1\\
0<|\ell^{\ast}|_\infty \leq q^{-\delta}|zd|_\infty^{-1}}} \ \sum\limits_{\substack{|h|_{\infty}\le q^M\\ 
h \equiv 0 \bmod{\ell^{\ast}}}} q^M + 
\sum\limits_{0\le j\le \mathcal{L}} q^{-j} \times\\ 
& \sum\limits_{\deg \alpha \le j}  
\sum\limits_{\substack{(\ell^{\ast},v^{\ast})=1\\
0<|\ell^{\ast}|_\infty \leq q^{-\delta}|zd|_\infty^{-1}}}\ \sum\limits_{\substack{|h|_{\infty}\le q^M\\ 
(f_k^\ast v^{\ast}+a)h \equiv \alpha v^{\ast}\bmod{\ell^{\ast}}}} |\ell^{\ast}|_{\infty},
\end{split}
\end{equation}
where we use $(v^{\ast},\ell^{\ast})=1$ and $(f_k^{\ast}v^{\ast}+a, v^{\ast})=1$. Further, 
\begin{equation} \label{firstsum}
\begin{split}
 \sum\limits_{\substack{(\ell^{\ast},v^{\ast})=1\\
0<|\ell^{\ast}|_\infty \leq q^{-\delta}|zd|_\infty^{-1}}} \ \sum\limits_{\substack{|h|_{\infty}\le q^M\\ 
h \equiv 0 \bmod{\ell^{\ast}}}} q^M \ll & \sum\limits_{\substack{(\ell^{\ast},v^{\ast})=1\\
0<|\ell^{\ast}|_\infty \leq q^{-\delta}|zd|_\infty^{-1}}} q^M\left(1+\frac{q^M}{|\ell^{\ast}|_\infty}\right)\\
\ll & q^{Q-2\delta}|v|_{\infty}|z|_{\infty}^{-1}|d|_{\infty}^{-2}+q^{2Q-2\delta}|v|_{\infty}^2|d|_{\infty}^{-2},
\end{split}
\end{equation}
where we use \eqref{Mrewrite}. 
Finally, we bound the triple sum over $\alpha$, $\ell^{\ast}$ and $h$ in the last line of \eqref{div}. We consider two
cases: If $(f_k^\ast v^{\ast}+a)h =\alpha v^{\ast}$, then the congruence 
$(f_k^\ast v^{\ast}+a)h \equiv \alpha v^{\ast} \bmod{\ell^{\ast}}$ 
is satisfied for every $\ell^{\ast}$. If $(f_k^\ast v^{\ast}+a)h \not= \alpha v^{\ast}$, then the above
congruence is equivalent to $\ell^{\ast}| (f_k^\ast v^{\ast}h+ah -\alpha v^{\ast})$. Hence,
\begin{equation} \label{secondsum}
\begin{split}
& \sum\limits_{\deg \alpha \le j}  
\sum\limits_{\substack{(\ell^{\ast},v^{\ast})=1\\
0<|\ell^{\ast}|_\infty \leq q^{-\delta}|zd|_\infty^{-1}}}\ \sum\limits_{\substack{|h|_{\infty}\le q^M\\ 
(f_k^\ast v^{\ast}+a)h \equiv \alpha v^{\ast}\bmod{\ell^{\ast}}}} |\ell^{\ast}|_{\infty}\\
= & \sum\limits_{\deg \alpha \le j}\
\sum\limits_{\substack{\deg h\le Q+\deg v -\delta-\deg d \\ (f_k^\ast v^{\ast}+a)h = \alpha v^{\ast}}}   
\sum\limits_{\substack{(\ell^{\ast},v^{\ast})=1\\
0<|\ell^{\ast}|_\infty \leq q^{-\delta}|zd|_\infty^{-1}}} |\ell^{\ast}|_{\infty}+\\
& \sum\limits_{\deg \alpha \le j}\
\sum\limits_{\substack{\deg h\le Q+\deg v -\delta-\deg d \\ (f_k^\ast v^{\ast}+a)h \not= \alpha v^{\ast}}}   
\sum\limits_{\substack{(\ell^{\ast},v^{\ast})=1\\
0<|\ell^{\ast}|_\infty \leq q^{-\delta}|zd|_\infty^{-1}\\
\ell^{\ast}| (f_k^\ast v^{\ast}h+ah -\alpha v^{\ast})}} |\ell^{\ast}|_{\infty}\\
\ll_q  & q^{-2\delta}|zd|_\infty^{-2} \sum\limits_{\deg \alpha \le j}\
\sum\limits_{\substack{\deg h\le Q+\deg v -\delta-\deg d \\ (f_k^\ast v^{\ast}+a)h = \alpha v^{\ast}}} 1 +\\
& 2^{Q+\mathcal{L}}q^{j+Q-2\delta}|v|_{\infty}|z|_\infty^{-1}|d|_{\infty}^{-2},
\end{split}
\end{equation}
where we use the estimates
$$
\sum\limits_{\substack{|\ell^{\ast}|_{\infty}\le X\\ 
\ell^{\ast}|n}} |\ell^{\ast}|_{\infty} \le X\sum\limits_{ 
\ell^{\ast}|n} 1 \ll_q 2^{\deg n}X 
$$
and 
$$
\deg (f_k^\ast v^{\ast}h+ah -\alpha v^{\ast})\le Q+\mathcal{L}.
$$
We further observe that
$$
\sum\limits_{\deg \alpha \le j}\
\sum\limits_{\substack{\deg h\le Q+\deg v -\delta-\deg d \\ (f_k^\ast v^{\ast}+a)h = 
\alpha v^{\ast}}} 1 \le q^{j+1}|f_k^\ast v^{\ast}+a|_{\infty}^{-1}
$$
because $(f_k^\ast v^{\ast}+a, v^{\ast})=1$. By \eqref{fkast}, we have
$$
|f_k^{\ast}v^{\ast}+a|_{\infty}= |f_k^{\ast}v^{\ast}|_{\infty} = |vz|_{\infty}^{-1}. 
$$
It follows that 
\begin{equation} \label{bart}
\sum\limits_{\deg \alpha \le j}\
\sum\limits_{\substack{\deg h\le Q+\deg v -\delta-\deg d \\ (f_k^\ast v^{\ast}+a)h = \alpha v^{\ast}}} 1 
\le q^{j+1}|vz|_{\infty}.
\end{equation}

Combining \eqref{div}, \eqref{firstsum}, \eqref{secondsum} and \eqref{bart}, we obtain
\begin{equation} \label{lisa}
\begin{split}
& \sum\limits_{\substack{(\ell^{\ast},v^{\ast})=1\\
0<|\ell^{\ast}|_\infty \leq q^{-\delta}|zd|_\infty^{-1}}}
\sum_{|h|_{\infty}\le q^M} \min\{q^M,\|Wh\|^{-1}\}\\ \le & 
q^{2Q-2\delta}|v|_{\infty}^2|d|_{\infty}^{-2} 
+ 2^{Q+\mathcal{L}}\mathcal{L}q^{Q-2\delta}|v|_{\infty}|z|_{\infty}^{-1}|d|_\infty^{-2}.
\end{split}
\end{equation}
Plugging this into \eqref{CSch} gives
\begin{equation} \label{Sigmad} 
\begin{split}
& \Sigma_d^2 \ll_q \left((Q+\mathcal{L})|v|_\infty^{-1}+q^{4Q-2\delta}|v|_{\infty}^3|z|_{\infty}^2|d|_{\infty}^{-2}\right)\times\\
& 
\left(q^{2Q-2\delta}|v|_{\infty}^2|d|_{\infty}^{-2} 
+ 2^{Q+\mathcal{L}}\mathcal{L}q^{Q-2\delta}|v|_{\infty}|z|_{\infty}^{-1}|d|_\infty^{-2}\right).
\end{split}
\end{equation}
Hence, the expression in \eqref{hardcase1} is bounded by 
\begin{equation} \label{hardcase3}
\begin{split}
\ll_q & (Q+\mathcal{L})^2 q^{\delta}|z|_{\infty}^{1/2}  \mathcal{L} 
\left((Q+\mathcal{L})|v|_\infty^{-1}+q^{4Q-2\delta}|v|_{\infty}^3|z|_{\infty}^2\right)^{1/2}\times\\
&  \left(q^{2Q-2\delta}|v|_{\infty}^2 
+ 2^{Q+\mathcal{L}}\mathcal{L}q^{Q-2\delta}|v|_{\infty}|z|_{\infty}^{-1}\right)^{1/2},
\end{split}
\end{equation}
where we use the bound 
$$
\sum\limits_{d|v} \frac{1}{|d|_\infty} \le \sum\limits_{0<|d|_{\infty}\le |v|_{\infty}} \frac{1}{|d|_\infty} \ll_q \log_q |v|_{\infty} \le 
\mathcal{L}. 
$$

\section{Third estimate for $P(x)$}
Combining \eqref{threecases} (total contribution to $P(x)$ of the simple cases) and \eqref{hardcase3} (total contribution to $P(x)$ of the critical case), and simplifying, we obtain the estimate
\begin{equation*}
\begin{split}
P\left(\frac{u}{v}+z\right)
\ll_q  & q^{\delta+Q}|z|_\infty + |v|_{\infty}^{1/2}+ (Q+\mathcal{L})^4 q^{\delta}|z|_{\infty}^{1/2} \times\\ &  
\left(|v|_\infty^{-1}+q^{4Q-2\delta}|v|_{\infty}^3|z|_{\infty}^2\right)^{1/2}\times\\
&  \left(q^{2Q-2\delta}|v|_{\infty}^2 
+ 2^{Q+\mathcal{L}}q^{Q-2\delta}|v|_{\infty}|z|_{\infty}^{-1}\right)^{1/2}.
\end{split}
\end{equation*}
Choosing $q^{\delta}$ as small as possible in \eqref{Eqn3}, i.e.
$$
\frac{q^{2Q-1}\Delta}{|z|_{\infty}}=\frac{q^{Q_0-1}\Delta}{|z|_{\infty}}\le q^{\delta}<
\frac{q^{Q_0}\Delta}{|z|_{\infty}}=\frac{q^{2Q}\Delta}{|z|_{\infty}}
$$
and using 
$$
|vz|_{\infty}\le \Delta^{1/2}
$$
and the (rough) bound
$$
(Q+\mathcal{L})^4\ll 2^{(Q+\mathcal{L})/2},
$$
we arrive at the following estimate for $P(x)$.

\begin{Proposition} \label{thirdPx} We have 
\begin{equation*}
\begin{split}
P\left(\frac{u}{v}+z\right)
\ll_q  & q^{3Q}\Delta+|v|_{\infty}^{1/2}+(Q+\mathcal{L})^4q^{Q}\Delta^{1/4}
+ 2^{Q+\mathcal{L}}q^{3Q/2}\Delta^{1/2}.
\end{split}
\end{equation*}
\end{Proposition}

\section{Proof of Theorem \ref{themainresult}}
Finally, we are ready to prove Theorem \ref{themainresult}. We 
use Proposition \ref{FirstPx} if $|v|_{\infty}> q^{Q}$ and Proposition \ref{thirdPx} if $|v|_{\infty}\le 
q^Q$ to get
\begin{equation} \label{burns}
\begin{split}
P\left(\frac{u}{v}+z\right)
\ll_q  & q^{3Q}\Delta+(Q+\mathcal{L})^4q^{Q}\Delta^{1/4}
+ 2^{Q+\mathcal{L}}\left(q^{Q/2}+q^{3Q/2}\Delta^{1/2}\right).
\end{split}
\end{equation}
Alternatively, we have the estimate
\begin{equation} \label{homer}
P\left(\frac{u}{v}+z\right) \ll_q  1+2^{Q+\mathcal{L}}\left(q^{2Q}\Delta^{1/2}+q^{3Q}\Delta\right)
\end{equation}
from Proposition \ref{secondPx}. If $q^{2Q}\le \Delta^{-1}\le q^{3Q}$, then \eqref{burns} gives the estimate
\begin{equation} \label{burns2}
\begin{split}
P\left(\frac{u}{v}+z\right)
\ll_q  & 2^{Q+\mathcal{L}}\left(q^{3Q}\Delta+q^{Q/2}\right),
\end{split}
\end{equation}
and if $q^{3Q}\le \Delta^{-1}\le q^{4Q}$, then \eqref{homer} gives the estimate
\begin{equation} \label{homer2}
P\left(\frac{u}{v}+z\right) \ll_q  2^{Q+\mathcal{L}}\left(q^{3Q}\Delta+q^{2Q}\Delta^{1/2}\right).
\end{equation}
We observe that 
$$
Q^{1/2}\le q^{2Q}\Delta^{1/2} \Leftrightarrow \Delta^{-1}\le q^{3Q}. 
$$
Hence, in the range $q^{2Q}\le \Delta^{-1}\le q^{4Q}$, we have
\begin{equation} \label{millhouse}
P\left(\frac{u}{v}+z\right) \ll_q  2^{Q+\mathcal{L}}\left(q^{3Q}\Delta+\min\left\{q^{Q/2},q^{2Q}\Delta^{1/2}\right\}\right).
\end{equation}
Now Theorem \ref{themainresult} follows from \eqref{millhouse} and Lemmas \ref{LS2} and \ref{K-P} upon taking $\Delta:=q^{-N}$. $\Box$

\end{document}